\documentclass[11pt,a4paper]{amsart}
\usepackage{a4wide}
\usepackage[latin1]{inputenc}
\usepackage{amsmath}
\usepackage{amsfonts}
\usepackage{amssymb}
\usepackage[dvipsnames]{xcolor}

\usepackage{graphicx}
\usepackage{amsthm}
\usepackage{amssymb}
\usepackage{cite}
\usepackage{mathrsfs}
\usepackage[backref=page]{hyperref}
\usepackage{paralist}
\usepackage{enumitem}
\usepackage{tikz}
\usetikzlibrary{matrix,chains}
\usepackage{dynkin-diagrams}

\begin{document}

	\newcommand{\n}{\mathbf{n}}
	\newcommand{\x}{\mathbf{x}}
	\newcommand{\h}{\mathbf{h}}
	\newcommand{\m}{\mathbf{m}}
	\newcommand{\Ph}{\Phi_}
	
	\newcommand{\bN}{\mathbf{N}}
	\newcommand{\bL}{\mathbf{L}}
	\newcommand{\bG}{\mathbf{G}}
	\newcommand{\bS}{\mathbf{S}}
	\newcommand{\bZ}{\mathbf{Z}}
	\newcommand{\bH}{\mathbf{H}}
	\newcommand{\B}{\mathbf{B}}
	\newcommand{\U}{\mathbf{U}}
	\newcommand{\K}{\mathbf{K}}
	\newcommand{\V}{\mathbf{V}}
	\newcommand{\T}{\mathbf{T}}
	\newcommand{\G}{\mathbf{G}}
	\newcommand{\Para}{\mathbf{P}}
	\newcommand{\Levi}{\mathbf{L}}
	\newcommand{\Y}{\mathbf{Y}}
	\newcommand{\X}{\mathbf{X}}
	\newcommand{\M}{\mathbf{M}}
	\newcommand{\pro}{\mathbf{prod}}
	\renewcommand{\o}{\overline}
	
	\newcommand{\Gtilde}{\mathbf{\tilde{G}}}
	\newcommand{\Ttilde}{\mathbf{\tilde{T}}}
	\newcommand{\Btilde}{\mathbf{\tilde{B}}}
	\newcommand{\Ltilde}{\mathbf{\tilde{L}}}
	\newcommand{\C}{\operatorname{C}}
	
	\newcommand{\N}{\operatorname{N}}
	\newcommand{\bl}{\operatorname{bl}}
	\newcommand{\Z}{\operatorname{Z}}
	\newcommand{\Gal}{\operatorname{Gal}}
	\newcommand{\modulo}{\operatorname{mod}}
	\newcommand{\kernel}{\operatorname{ker}}
	\newcommand{\Irr}{\operatorname{Irr}}
	\newcommand{\D}{\operatorname{D}}
	\newcommand{\I}{\operatorname{I}}
	\newcommand{\GL}{\operatorname{GL}}
	\newcommand{\SL}{\operatorname{SL}}
	\newcommand{\W}{\mathbf{W}}
	\newcommand{\R}{\operatorname{R}}
	\newcommand{\Br}{\operatorname{Br}}
	\newcommand{\Aut}{\operatorname{Aut}}
	\newcommand{\Out}{\operatorname{Out}}
	\newcommand{\End}{\operatorname{End}}
	\newcommand{\Ind}{\operatorname{Ind}}
	\newcommand{\Res}{\operatorname{Res}}
	\newcommand{\br}{\operatorname{br}}
	\newcommand{\Hom}{\operatorname{Hom}}
	\newcommand{\Endo}{\operatorname{End}}
	\newcommand{\Ho}{\operatorname{H}}
	\newcommand{\Tr}{\operatorname{Tr}}
	\newcommand{\opp}{\operatorname{opp}}
	\newcommand{\ssc}{\operatorname{sc}}
	\newcommand{\ad}{\operatorname{ad}}
	\newcommand{\Lin}{\operatorname{Lin}}
	\newcommand{\odd}{\operatorname{odd}}
	\newcommand{\even}{\operatorname{even}}
	
	\newcommand{\cind}{\operatorname{co-ind}}
	\newcommand{\ind}{\operatorname{ind}}
	\newcommand{\CC}{{\mathbb{C}}}	
	\newcommand{\QQ}{{\mathbb{Q}}}	
	\newcommand{\fF}{{\mathfrak F}}
	\newcommand{\la}{{\lambda}}
    \newcommand{\La}{{\Lambda}}
	\newcommand{\al}{{\alpha}}
	\newcommand{\cH}{{\mathcal H}}
	\newcommand{\tw}[1]{{}^#1\!}

    \newcommand{\wt}{\widetilde}
    \newcommand{\galh}{\mathcal{H}}   \newcommand{\gal}{\mathcal{G}}
    \newcommand{\type}{\operatorname}

\newcommand{\tA}{\type{A}}
\newcommand{\tB}{\type{B}}
\newcommand{\tC}{\type{C}}
\newcommand{\tD}{\type{D}}
\newcommand{\tE}{\type{E}}
\newcommand{\tF}{\type{F}}
\newcommand{\tG}{\type{G}}

\newcommand{\subL}{\mathfrak{L}}

    \newcommand{\Syl}{\mathrm{Syl}}

    \theoremstyle{definition}
	\newtheorem{definition}{Definition}[section]
	\newtheorem{notation}[definition]{Notation}
	\newtheorem{construction}[definition]{Construction}
	\newtheorem{remark}[definition]{Remark}
	\newtheorem{example}[definition]{Example}

    \theoremstyle{theorem}
	
	\newtheorem{theorem}[definition]{Theorem}
	\newtheorem{lemma}[definition]{Lemma}
	\newtheorem{question}{Question}
	\newtheorem{corollary}[definition]{Corollary}
	\newtheorem{proposition}[definition]{Proposition}
	\newtheorem{conjecture}[definition]{Conjecture}
	\newtheorem{assumption}[definition]{Assumption}
	\newtheorem{hypothesis}[definition]{Hypothesis}
	\newtheorem{maintheorem}[definition]{Main Theorem}

	\newtheorem{theo}{Theorem}
	\newtheorem{conj}[theo]{Conjecture}
	\newtheorem{cor}[theo]{Corollary}

	\renewcommand{\thetheo}{\Alph{theo}}
	\renewcommand{\theconj}{\Alph{conj}}
	\renewcommand{\thecor}{\Alph{cor}}

	\newcommand{\ov}{\overline }
	
	\def\oo#1{\overline{\overline{#1}}}

    \newcommand{\mandicomment}{\textcolor{blue}}
    \newcommand{\lucascomment}{\textcolor{OliveGreen}}

\title{The Isaacs--Navarro Galois Conjecture}

\author{L. Ruhstorfer}
\address[L. Ruhstorfer]{{School of Mathematics and Natural Sciences}, {University of Wuppertal}, {Gau{\ss}str. 20,
		42119 Wuppertal, Germany}}
\email{ruhstorfer@uni-wuppertal.de}
\author{A. A. Schaeffer Fry}
\address[A. A. Schaeffer Fry]{{Department of Mathematics}, {University of Denver}, {Denver, CO 80210, USA}}
\email{mandi.schaefferfry@du.edu}

\thanks{The authors gratefully acknowledge support from a SQuaRE at the American Institute of Mathematics and they thank AIM for their generosity and providing a supportive and mathematically rich environment. The second-named author is partially supported by a CAREER grant through the U.S. National Science Foundation, Award No. DMS-2439897. Some of this research was conducted in the framework of the research
 training group GRK 2240: Algebro-Geometric Methods in Algebra, Arithmetic and Topology,
 funded by the DFG.
 The authors thank G. Malle, J.M. Mart{\'i}nez, G. Navarro, N. Rizo,  and C. Vallejo for comments on earlier versions of the manuscript}

\maketitle

\begin{center}
 \emph{   Dedicated to Gabriel Navarro for the occasion of his 60th birthday.}
\end{center}

\begin{abstract}
    We prove the original Galois refinement of the McKay conjecture, proposed by Isaacs--Navarro in \cite{IN}, providing an important subcase of the celebrated McKay--Navarro conjecture with several local-global consequences.
    
\end{abstract}

\section{Introduction}

For over a half-century, the study of representations of finite groups has been heavily influenced by the McKay conjecture (now a theorem), which says that a bijection exists between the set $\Irr_{\ell'}(G)$  of irreducible characters of a finite group $G$ with degree relatively prime to a prime $\ell$ and the corresponding set $\Irr_{\ell'}(\N_G(D))$ for the normalizer of a Sylow $\ell$-subgroup $D$ of $G$. In a culmination of much work on the topic, the proof of the McKay conjecture was recently completed by Cabanes and Sp{\"a}th in \cite{CS25}.

During the same time, the fields of values $\QQ(\chi):=\QQ(\{\chi(g)\mid g\in G\})$ for $\chi\in\Irr(G)$ have been seen to be valuable number-theoretic features and have been shown to relate to the structure of $G$ in numerous ways. For example, understanding these fields of values has strong implications for questions of Brauer, such as his Problem 12 asking what information can be obtained about a Sylow $\ell$-subgroup $D$ of $G$ from the character table of $G$. In particular, some of the results presented here give answers to this question.

The field $\QQ(\chi)$ lies in $\QQ(e^{2\pi i/|G|})\subset \QQ^{\mathrm{ab}}$, and the Galois group  $\gal:=\mathrm{Gal}(\QQ^{\mathrm{ab}}/\QQ)$ acts naturally on  $\Irr(G)$. 
In the last quarter-century, there has been much study of versions of the McKay conjecture that consider the role of these fields of values and the action of $\gal$, positing stronger bijections predicting relationships between fields of values of the $\ell'$-degree characters in a McKay bijection.  The focus of the present paper is the original of these Galois refinements of the McKay conjecture, found in \cite[Conj.~C]{IN}, which we will call the \emph{Isaacs--Navarro Galois Conjecture}.
Given a fixed prime $\ell$, we let $\mathcal{H}_0\leq\gal$ denote the subgroup consisting of all Galois automorphisms $\sigma \in \gal$ that act trivially on $\ell'$-roots of unity
and have $\ell$-power order. Our main result is the following, proving the Isaacs--Navarro Galois Conjecture:

\begin{theo}\label{thm:main}
    Let $G$ be a finite group. Suppose that $\ell$ is a prime dividing $|G|$, and let $D\in\Syl_\ell(G)$.  Then there exists an $\mathcal{H}_0$-equivariant bijection $\Irr_{\ell'}(G) \to \Irr_{\ell'}(\N_G(D))$.
\end{theo}
   
The Isaacs--Navarro Galois Conjecture has a number of interesting consequences, many of which had not yet been proven for odd primes but are now implied by our Theorem \ref{thm:main}. 
For example, as noted already in \cite{IN}, the Isaacs--Navarro Galois conjecture implies a conjecture on exponents of abelianizations of Sylow subgroups, which was the focus of \cite{NT19}. 
As a corollary to our main result, we obtain that statement. Given an integer $e\geq 1$, let $\sigma_e\in\galh_0$ be the element sending $\ell$-power order roots of unity $\xi$ to $\xi^{1+\ell^e}$.

\begin{cor}\label{cor:NTconj}
Let $G$ be a finite group, $\ell$ a prime, and let $D\in\Syl_\ell(G)$. Then the following are equivalent:
\begin{itemize}
    \item $\mathrm{Exp}(D/D')\leq \ell^e$;
    \item all characters in $\Irr_{\ell'}(G)$ are fixed by $\sigma_e$; and
    \item all characters in $\Irr_{\ell'}(B_0(G))$ are fixed by $\sigma_e$, where $B_0(G)$ denotes the principal $\ell$-block of $G$.
\end{itemize} 
(In particular,  \cite[Conj.~A]{NT19} holds for all finite groups and all primes.)
\end{cor}

 While the equivalence of the first two items will follow from Theorem \ref{thm:main}, we note that in \cite{NT19}, it was proved that the third item of Corollary \ref{cor:NTconj} implies the first. That the first item implies the second was reduced to a problem on simple groups in \cite{NT19}, which was used by Malle in \cite{Malle19} to complete the proof of the statement of Corollary \ref{cor:NTconj} for $\ell=2$. (Note that it also now follows in that case from the main result of \cite{RSF25}.) In contrast, despite strong further work on the problem and related problems  (see, e.g. \cite{NT21, hung}), the case of $\ell$ odd was more elusive, remaining open until now.

\medskip

As a consequence of Theorem \ref{thm:main}, we also obtain the main conjecture of N. Hung from \cite{hung}. For a character $\chi\in\Irr(G)$, we write $\mathrm{lev}(\chi)$ for the $\ell$-rationality level of $\chi$, as defined in \cite{hung}. This number is closely tied to the behavior of $\sigma_e$, and gives a measure of how far a character is from being $\ell$-rational. Recall here that $\chi\in\Irr(G)$ is called $\ell$-rational if the conductor $c(\chi)$ is not divisible by $\ell$, where the conductor of $\chi$ is the smallest integer $c:=c(\chi)$ such that $\QQ(\chi)\subseteq \QQ(e^{2\pi i/c})$. The $\ell$-rationality level $\mathrm{lev}(\chi)$ is then the largest $e\geq 0$ such that $\ell^e$ divides $c(\chi)$. That is, $\mathrm{lev}(\chi)=\log_\ell(c(\chi)_\ell)$. As noted in \cite[Sec.~2]{hung}, if $\chi$ is not $\ell$-rational (that is, $\mathrm{lev}(\chi)>0$), then $\mathrm{lev}(\chi)$ is the smallest $e$ such that $\chi$ is $\sigma_e$-stable; further,
the Isaacs--Navarro Galois conjecture (hence, our Theorem \ref{thm:main}) implies that the number of characters in $\Irr_{\ell'}(G)$ with a given $\ell$-rationality level larger than $1$ is the same as the corresponding number in $\Irr_{\ell'}(\N_G(D))$ for $D\in\Syl_\ell(G)$.  The following was conjectured in \cite{hung}.

\begin{cor}\label{cor:Hungconj}
Let $\ell$ be a prime and $G$ a finite group.
Let $e\geq 2$ be an integer. Then the following hold:
\begin{itemize} 
\item If there is some $\chi\in\Irr_{\ell'}(G)$ with $\mathrm{lev}(\chi)=e$, then there exist characters in $\Irr_{\ell'}(G)$ with every $\ell$-rationality level between $2$ and $e$.

\item Let $M\leq G$ with $\ell\nmid [G:M]$. Then $\Irr_{\ell'}(G)$ contains a character $\chi$ with $\mathrm{lev}(\chi)=e$ if and only if $\Irr_{\ell'}(M)$ does.
\end{itemize}
(That is, \cite[Conjs.~1.1 and ~2.3]{hung} hold for all finite groups and all primes.)
\end{cor}

We remark that the second point of Corollary \ref{cor:Hungconj} follows from the first and Corollary \ref{cor:NTconj}. It is also worth noting that, again, the case of $\ell=2$ in Corollary \ref{cor:Hungconj} was more approachable, proved already in \cite{hung} (and later also following independently from \cite{RSF25}), while the case that $\ell$ is odd remained a challenge. 

\medskip

In \cite{N04}, Navarro extended the Isaacs--Navarro Galois Conjecture to use a larger group $\galh_\ell$ containing $\galh_0$. (Turull  \cite{Tur08} later related this to the $\ell$-adic numbers and suggested a version relating Schur indices.)  Namely, $\galh_\ell\leq \gal$ is the subgroup comprised of those $\sigma\in\gal$ satisfying that there is some $e\geq 0$ such that $\sigma$ acts on $\ell'$-roots of unity by $\zeta\mapsto \zeta^{\ell^e}.$ Navarro's conjecture requiring an $\galh_\ell$-equivariant bijection, in place of our $\galh_0$-equivariant bijection, is what is generally known as the Galois--McKay or McKay--Navarro conjecture \cite{N04}. The latter was reduced to a problem on simple groups by Navarro--Sp{\"a}th--Vallejo in \cite{NSV20} and  was recently proved by the current authors in the case $\ell=2$ in \cite{RSF25} using this reduction. Our proof of Theorem \ref{thm:main} (and hence Corollaries \ref{cor:NTconj} and \ref{cor:Hungconj}) will also use this reduction, together with recent work of the authors, Sp{\"a}th, and Taylor in \cite{RSST} and building on the previous work toward the McKay--Navarro conjecture and, of course, the proof of the McKay conjecture.

The remainder of the paper is organized as follows. In Section \ref{sec:inductiveconds}, we briefly discuss the inductive conditions from \cite{NSV20}, which we will use to complete the proof of Theorem \ref{thm:main}. In Section \ref{sec:general}, we make some observations regarding extensions that will be useful throughout. In Section \ref{sec:nonlie}, we reduce our problem to the case of groups of Lie type in non-defining characteristic, and in Section \ref{sec:lie}, we reduce further to criteria specific to those groups using the work of \cite{RSST}. Section \ref{sec:toward} is devoted to those criteria, and  there we finish the case of groups of Lie type $\tA$ as well as the case that a group of Lie type is defined over $\mathbb{F}_q$ where the order of $q$ modulo $\ell$ is a so-called regular number. The final proof of Theorem \ref{thm:main} appears in Section \ref{sec:nonreg}, using an induction argument to complete the proof for non-regular numbers. Finally, in Section \ref{sec:corollaries}, we make some final remarks on implications of Theorem \ref{thm:main}, in particular completing the discussion of Corollaries \ref{cor:NTconj} and \ref{cor:Hungconj}.

\section{The inductive Galois--McKay condition}\label{sec:inductiveconds}

The McKay--Navarro conjecture was reduced in  \cite[Thm.~A]{NSV20} to a problem on (quasi-) simple groups, and we note that the reduction also works if we replace the group $\galh_\ell$ with any of its subgroups. In particular, taking the subgroup $\galh_0$ as defined in the introduction, we have the following:

\begin{theorem}\label{thm:reduction}
The Isaacs--Navarro Galois conjecture holds if every finite non-abelian simple group satisfies the inductive Galois--McKay condition \cite[Def.~3.1]{NSV20} with respect to the subgroup $\mathcal{H}_0$. 
\end{theorem}

\begin{definition}
We will say that \emph{the inductive Isaacs--Navarro condition} holds for a simple group $S$ and the prime $\ell$ if the inductive Galois--McKay condition \cite[Def.~3.1]{NSV20} holds for $S$ with respect to the subgroup $\galh_0\leq \galh_\ell$. If the inductive Isaacs--Navarro condition holds for $S$ for all primes $\ell$, we simply say that the inductive Isaacs--Navarro condition holds for $S$.
\end{definition}

It is worth remarking that the inductive condition for a simple group $S$ is actually a condition on its universal covering group. 
Thanks to the work in \cite{RSST}, we will mostly be able to deal with a refined set of conditions (see Theorem \ref{thm:RSST} below). For this reason, we do not explicitly redefine the inductive Galois--McKay condition here, but instead refer the interested reader to \cite[Def.~3.1]{NSV20}.

The following notation and definitions will play an important role throughout.

\begin{definition}
For groups $X\leq Y$, we will use the usual notation $\Res_X^Y(\chi)$ to denote the restriction to $X$ of a character $\chi$ of $Y$.   
For $G\lhd A$ and subsets $\mathfrak{G}\subseteq \Irr(G)$ and $\mathfrak{A}\subseteq\Irr(A)$, we write $\Irr(A\mid\mathfrak{G})$ for the irreducible characters of $A$ whose irreducible constituents on restriction to $G$ lie in $\mathfrak{G}$ and $\Irr(G\mid\mathfrak{A})$ for the irreducible constituents of characters in $\mathfrak{A}$ on restriction to $G$. Further, an \textit{extension map} with respect to $G\lhd A$ for $\mathfrak{G}$ is a map $\Lambda:\mathfrak{G}\longrightarrow \bigcup_{G\leq I \leq A} \Irr(I)$ such that for every $\chi\in\mathfrak{G}$, $\Lambda(\chi)$ is an extension of $\chi$ to the stabilizer $A_\chi$ of $\chi$ in $A$. 

Now suppose that $B$ is a group acting on $\Irr(A_\chi)$ for each $\chi\in\Irr(G)$. Given an extension map $\Lambda$ with respect to $G\lhd A$, we say that $\Lambda$ is \emph{$B$-equivariant} if $\Lambda(\chi^\sigma)=\Lambda(\chi)^\sigma$ for every $\sigma\in B$ and $\chi\in\Irr(G)$. 

If $\chi\in\Irr(G)$ extends to a character $\wt \chi$ of $A_\chi$ and $\chi$ is invariant under $\alpha\in B$, then there exists a unique linear character $\mu\in \Irr(A_\chi/G)$ such that $\wt\chi^\alpha=\wt\chi\mu$,
    	by Gallagher's theorem \cite[Cor.~6.17]{Isa}. In this situation, we will write $[\wt\chi, \alpha]:=\mu$ for this character.
\end{definition}

In constructing extension maps, the following definitions will also be used throughout:

\begin{definition}
If $\chi\in \Irr(Y)$ is an irreducible character of a finite group $Y,$ we denote by $\det(\chi)$ its determinantal character.  Moreover, $o(\chi)$ denotes the order of the determinantal character, see the remarks before \cite[Lemma 6.24]{Isa}.
\end{definition}

\section{General Observations on Character Extensions}\label{sec:general}

We often use the following lemma about gluing extensions:

\begin{lemma}\label{lem:gluing}
    Let $X_1 \lhd Y$ and $X_2 \leq Y$ such that $Y=X_1 X_2$. Assume that $\vartheta_1 \in \Irr(X_1)$ is $Y$-invariant and $\vartheta_2 \in \Irr(X_2)$ is such that $\Res_{X_1 \cap X_2}^{X_2}(\vartheta_2)$ is irreducible and coincides with $\Res_{X_1 \cap X_2}^{X_1}(\vartheta_1)$. Then there exists a unique character $\vartheta \in \Irr(Y)$ which extends both $\vartheta_1$ and $\vartheta_2$.
\end{lemma}
\begin{proof}
See, for example, \cite[Lem.~4.1]{spath10}. (This is also a case of a more general statement by Isaacs---see \cite[Lem.~6.8]{Nav18}.)
\end{proof}

This allows us to simplify many of the inductive conditions. For example, we have the following lemma:

\begin{lemma}\label{lem:extnsinitial_var}
    Let $X \lhd Y \lhd \hat{Y}$ with $X \lhd \hat{Y}$ 
    such that $Y/X$ has a normal Sylow $\ell$-subgroup and $\hat{Y}/Y$ is a (possibly trivial) $\ell$-group. Assume that $\psi \in \Irr_{\ell'}(X)$ extends to $Y_\psi$. Suppose further that at least one of the following holds:
    \begin{itemize}
        \item[(i)] $\ell\nmid o(\psi)$ or
        \item[(ii)] there exists $\hat{K} \leq \hat{Y}$ such that $\hat{Y} = X\hat{K}$ and $K:=\hat{K} \cap X$ is an $\ell'$-group.
    \end{itemize}
    Then there exists an extension $\chi\in\Irr(Y_\psi)$ of $\psi$ such that the stabilizers $(\hat{Y} \times \mathcal{H}_{0})_{\chi}=(\hat{Y} \times \mathcal{H}_{0})_{\psi}$ are the same.
\end{lemma}

\begin{proof}

Let $X\leq P \leq Y$ such that $P/X\lhd Y_\psi/X$ is the normal Sylow $\ell$-subgroup of $Y_\psi/X$, and let $\lambda:=\det(\psi)$. 

We first claim that $\psi$ has a $\hat{Y}_\psi$-invariant extension $\chi_P$ to $P$ such that $(\hat{Y}\times\mathcal{H}_{0})_{\chi_P}= (\hat{Y}\times\mathcal{H}_{0})_{\psi}$. If $\la$ has order prime to $\ell$, then according to \cite[Lem.~8.16]{Isa}, there exists a unique extension $\chi_P$ of $\psi$ to $P$ such that 
$o(\chi_P)=o(\psi)$.
In particular, $(\hat{Y}\times\mathcal{H}_{0})_{\chi_P}= (\hat{Y}\times\mathcal{H}_{0})_{\psi}$, and $\chi_P$ is $\hat{Y}_\psi$-invariant, giving the claim in this case. 

Now suppose that assumption (ii) is satisfied. Note that $\la_0:=\Res^{X}_{K_\psi}(\la)$ has an extension to $(\hat{K} \cap Y)_\psi$, since the character $\psi$ extends to $Y_\psi=X (\hat{K} \cap Y)_\psi$ and $\la=\det(\psi)$.  As $K_\psi$ is an $\ell'$-group and $\la_0$ is a linear character, we can find an extension of $\la_0$ to $(\hat{K} \cap Y)_\psi$ such that all elements of $(\hat{K} \cap Y)_\psi$ of $\ell$-power order are in the kernel of the extension. By Gallagher's lemma \cite[Cor.~6.17]{Isa}, the number of extensions with this property is prime to $\ell$. On the other hand, note that both $(\hat{K}  \times \mathcal{H}_0)_\psi/\hat{K}_\psi$ and $\hat{K}_\psi/(\hat{K} \cap Y)_\psi$ are $\ell$-groups. Then by coprimality, there exists such an extension $\hat{\la}_0$ to $(\hat{K} \cap Y)_\psi$ which is $(\hat{K} \times \mathcal{H}_0)_\psi$-stable. By applying Lemma \ref{lem:gluing}, we obtain the unique character $\hat{\la}$ of $Y_\psi=X (\hat{K} \cap Y)_\psi$ that extends both $\la$ and $\hat{\la}_0$, which satisfies $(\hat{K} \times \mathcal{H}_{0})_{\psi}\leq(\hat{K} \times \mathcal{H}_{0})_{\hat{\la}}$.
In particular, the restriction $\Res_P^{Y_\psi}(\hat{\la})$ of this extension to $P$ is $(\hat{Y} \times \mathcal{H}_0)_\psi$-stable (as $\hat{Y}=X \hat{K}$).
Then by \cite[Lem.~6.24]{Isa}, there is a unique extension $\chi_P$ of $\psi$ to $P$ such that  $\det(\chi_P)=\Res_P^{Y_\psi}(\hat{\la})$. This forces again that $(\hat{Y} \times \mathcal{H}_{0})_{\chi_P}= (\hat{Y}\times\mathcal{H}_{0})_{\psi}$ and $\chi_P$ is $\hat{Y}_\psi$-invariant, completing the claim.

Now, by assumption, the character $\psi$ has an extension $\chi_0$ to $Y_\psi$. Let $\chi_0'$ be its restriction to $P$.
By Gallagher's lemma \cite[Cor.~6.17]{Isa}, there exists a unique linear character $\nu \in \Irr(P/X)$ such that $\chi_P=\chi_0' \nu$. Now note that $P/X=XK/X \cong K/ (X \cap K)$ is an $\ell$-group. As both $\chi_0'$ and $\chi_P$ are $Y_\psi$-invariant, the uniqueness of $\nu$ implies that $\nu$ is also $Y_\psi$-invariant. 
As $o(\nu)$ is an $\ell$-power and $Y_\psi /P$ is an $\ell'$-group, we can now apply \cite[Cor.~6.27]{Isa} to find a unique extension $\hat{\nu}$ of $\nu$ to $Y_\psi$ such that $o(\hat{\nu})=o(\nu)$. It thus follows that $\chi_0 \hat{\nu}$ is an extension of $\chi_0' \nu=\chi_P$ to $Y_\psi$. 
But also by Gallagher's lemma, the number of extensions of $\chi_P$ to $Y_\psi$ is again an $\ell'$-number. Hence, by coprimality and using the fact that $(\hat{Y} \times \mathcal{H}_{0})_{\psi}/Y_{\psi}$ is an $\ell$-group (as both $(\hat{Y} \times \mathcal{H}_{0})_{\psi}/\hat{Y}_\psi$ and $\hat{Y}_\psi/Y_\psi$ are $\ell$-groups), there exists a $(\hat{Y} \times \mathcal{H}_{0})_{\psi}$-stable extension as desired.
\end{proof}

The second assumption in Lemma \ref{lem:extnsinitial_var} implies a nice splitting property of the Sylow $\ell$-subgroup:

\begin{remark}\label{rem:split sylow}
Let $X \lhd Y$ and suppose that there exists $K \leq Y$ such that $KX/X \in \Syl_\ell(Y/X)$ and $K \cap X$ is an $\ell'$-group. Then for $\hat{D} \in \Syl_\ell(Y)$ there exists $E_0 \leq K$ such that $\hat{D}=D \rtimes E_0$ with $D:=\hat{D} \cap X$.
\end{remark}

\begin{proof}
Since all Sylow $\ell$-subgroups of $Y$ are conjugate, it suffices to prove the claim for a fixed Sylow $\ell$-subgroup.
Let $E_0 \in \Syl_\ell(K)$. Then there is $\hat{D} \in \Syl_\ell(Y)$ with $\hat{D} \cap K =E_0$. Denote $D:=\hat{D} \cap X \in \Syl_\ell(X)$. Since $E_0 \cap D \leq E_0 \cap X\leq K\cap X$ is an $\ell'$-group, it follows that $E_0\cap D=1$ and  $\hat{D}=D \rtimes E_0$. 
\end{proof}

We draw a first consequence from Lemma \ref{lem:extnsinitial_var}, which will allow us to conclude the inductive Isaacs--Navarro condition for simple groups which are not of Lie type in the next section.

\begin{corollary}\label{cor:cyclicout}
Let $S$ be a non-abelian simple group such that $\mathrm{Out}(S)$ is cyclic, and let $G$ be the universal $\ell'$-covering group of $S$. If   there exists an $(\mathcal{H}_0 \times \mathrm{Aut}(G)_D)$-equivariant bijection $\Irr_{\ell'}(G) \to \Irr_{\ell'}(\N_G(D))$ with $D\in\Syl_\ell(G)$, then $S$ satisfies the inductive Isaacs--Navarro condition for $\ell$.
\end{corollary}

\begin{proof}
First, note that it suffices to prove the inductive conditions for the $\ell'$-cover, by \cite[Lem.~5.1]{Joh22b}.
Let $a\in \Aut(G)$ induce the full cyclic group of outer automorphisms on $G$. By conjugacy of Sylow $\ell$-subgroups, we find some $g \in G$ such that $ga$ normalizes $D$, so by replacing $a$ by $ga$ we can assume that $a$ normalizes $D$. Then consider the group $A:=G\rtimes \langle a\rangle$. 
   
    Hence, the assumptions of Lemma \ref{lem:extnsinitial_var} are satisfied for $X\in\{G, \N_G(D)\}$ and $Y=\hat Y=X\rtimes \langle a\rangle$. 
     In particular, every character $\chi$ of $G$ or $\N_G(D)$ has an extension $\hat \chi$ to its inertia group in $A$ resp. $\N_A(D)$ with $(A\times\galh_0)_\chi=(A\times\galh_0)_{\hat\chi}$, resp. $(\N_A(D)\times \galh_0)_\chi=(\N_A(D)\times \galh_0)_{\hat\chi}$, using Lemma \ref{lem:extnsinitial_var}. By \cite[Lem.~4.6]{birtethesis}, these extensions can further be chosen to have trivial associated scalars on $\C_A(G)$. (Note that the result in loc. cit. is stated for the full $\galh_\ell$, but follows exactly the same way when restricting to $\galh_0$.) Together with the assumption of the equivariant bijection, this gives the $\galh_0$-triple condition \cite[Def.~1.5]{NSV20}, and hence the inductive condition from \cite[Def.~3.1]{NSV20} with respect to $\galh_0$.
    \end{proof}

\section{
Groups not of Lie type and groups of  Lie type in defining characteristic}\label{sec:nonlie}

We next record several previous results regarding the inductive McKay--Navarro condition.

\begin{lemma}\label{lem:nonlie}
    The inductive McKay--Navarro condition holds for the simple group $S$ and prime $\ell$ when $(S, \ell)$ is as in any of the following situations:
    \begin{enumerate}
        \item $S$ is a simple group of Lie type defined in characteristic $p=\ell$;
        \item $S$ is a group of Lie type with exceptional Schur multiplier and $\ell$ is any prime; 
        \item $S$ is a simple Suzuki or Ree group (including the Tits group $\tw{2}\type{F}_4(2)'$) and $\ell$ is any prime;
        \item $\ell$ is odd and $S$ is a simple group of Lie type defined in characteristic $p\neq \ell$ such that the Schur covering group $G$ is one of the exceptions in \cite[Thm.~5.14]{Ma07} with a non-generic Sylow normalizer; or
        \item $S$ is any nonabelian simple group and $\ell=2$.
    \end{enumerate}
\end{lemma}

\begin{proof}
 Part (1) is the main result of \cite{Ruh21}, with finitely many exceptions completed in \cite{Joh22b}. Parts (2)-(4) are completed in \cite{birtethesis},  also appearing in \cite[Thm.~A, Prop.~6.4]{Joh24}. Part (5) is completed in \cite{RSF25}, building upon the above results and \cite{SF22, RSF22}.
\end{proof}

We next prove the inductive Isaacs--Navarro condition in the case of alternating groups, which were proved to satisfy the McKay--Navarro conjecture in \cite{BN21}.

\begin{lemma}\label{lem:altspor}
    The inductive Isaacs--Navarro condition holds for the simple alternating groups $A_n$ and the sporadic simple groups.
\end{lemma}

\begin{proof}
 From Lemma \ref{lem:nonlie}, we may assume $\ell$ is odd.  Note that when $S\neq A_6$,  we have $|\mathrm{Out}(S)|\leq 2$, and  $\mathrm{Out}(A_6)\cong C_2^2$. Let $G$ be the universal $\ell'$-covering group of $S$.

First, as pointed out in \cite[Thm.~5.2]{HSF24}, we have every $\chi\in\Irr_{\ell'}(G)$ has $\ell$-rationality level $\mathrm{lev}(\chi)\leq 1$, and hence these are fixed by $\galh_0$.  
Note that $P/[P,P]$ is elementary abelian for  $P\in\Syl_\ell(G)$ (see, e.g., \cite[Lem.~3.3, 3.4]{NT16}), so that $\Phi(P)=[P,P]$. 
In particular, we see $\Irr(\N_G(P)/\Phi(P))=\Irr(\N_G(P)/[P,P])=\Irr_{\ell'}(\N_G(P))$. Then applying \cite[Lem.~5.3]{MMV25}, we then have $\Irr_{\ell'}(\N_G(P))=\Irr_{\ell',\galh_0}(\N_G(P))$. That is, every $\ell'$-degree irreducible character of $\N_G(P)$ is also $\galh_0$-invariant. Since these groups also satisfy the inductive McKay conditions (see \cite[Thm.~3.1]{Malle08b}), it follows that there is an $(\galh_0\times\Aut(G)_P)$-equivariant bijection $\Omega\colon\Irr_{\ell'}(G)\rightarrow \Irr_{\ell'}(\N_G(P))$.  By Corollary \ref{cor:cyclicout}, this means the Isaacs--Navarro condition holds for $S$ if $S\neq A_6$.

Now let $S=A_6$. Then we may assume that $\ell=5$ by  \cite[Prop.~5.13]{birtethesis}.    
For characters $\chi\in\Irr_{\ell'}(G)$ with $|\Out(S)_\chi|\leq 2$, we may apply the considerations in Corollary \ref{cor:cyclicout}, so we assume that $\chi\in\Irr_{\ell'}(G)$ is stable under $\Out(S)$. By \cite[Lem.~4.4]{birtethesis}, $\chi$ extends to an $(\Aut(G)_P\times \galh_0)_\chi$-invariant character $\theta$ of $X:=G\rtimes\mathrm{Inn}(G)_P$, and similar for $\Omega(\chi)$. Applying Lemma \ref{lem:extnsinitial_var} and recalling that $\Aut(G)/\mathrm{Inn}(G)$ is a $2$-group, it then suffices to know that these extensions extend further to $Y:=(G\rtimes \Aut(G)_P)_\chi$ and $(\N_G(P)\rtimes\Aut(G)_P)_\chi$, respectively. In this situation, $Y/X$ is Klein-four, and we may apply Lemma \ref{lem:gluing} to see the statement.
\end{proof}

\section{Groups of Lie type in Non-defining characteristic}\label{sec:lie}

Thanks to Section \ref{sec:nonlie}, we are left to consider the case that $\ell$ is odd and that $S$ is a simple group of Lie type in non-defining characteristic with generic Schur multiplier and ``generic" Sylow-normalizer structure, in the sense of \cite[Thm.~5.14]{Ma07}. In particular, the universal covering group of $S$ is of the form $G=\bG^F$ where $(\bG, F)$ is a finite reductive group of simply connected type.

\subsection{Notation}\label{not:lie type}

Let $(\bG, F)$ be a finite reductive group such that $\bG$ is simple of simply connected type and $F$ is a Frobenius endomorphism defining an $\mathbb{F}_q$-structure on $\G$, where $q$ is a power of a prime $p$. We let $E(\G^F)$ denote the group of field and graph automorphisms of $\G^F$ as defined in \cite[Sec.~2.C]{CS25} and we set $E:=E(\G^F)$. We can then write $E=\langle \Gamma, F_p\rangle$, where $\Gamma$ is a group of graph automorphisms and $F_p$ is a standard Frobenius induced by the map $x\mapsto x^p$ on $\overline{\mathbb{F}}_p$.

We assume that $\ell$ is an odd prime with $\ell \neq p$ and let $d:=d_\ell(q)$ be the order of $q$ modulo $\ell$.
Let $\bS$ be a Sylow $d$-torus of $(\bG,F)$, as defined e.g. in \cite[3.5.6]{GM20}. Let $\bG\hookrightarrow\wt\bG$ be a regular embedding (see e.g. \cite[Sec.~1.7]{GM20}) and write $G:=\bG^F$ and $\wt{G}:=\wt\bG^F$. Let $\wt N:=\N_{\wt{G}}(\bS)$,  $N:=\wt N\cap G=\N_{G}(\bS)$, $\wt\bL:=\C_{\wt\bG}(\bS)$, $\bL:=\C_{\bG}(\bS)$, $\wt L:=\C_{\wt{G}}(\bS)=\wt\bL^F$, and $L:=\wt L\cap \G=\C_{G}(\bS)=\bL^F$. Further, let $\widehat N:=\N_{GE}(\Levi)$.

As defined in \cite{RSST}, we let 
\[\mathfrak{C}:=\{\wt\la\in\Irr(\wt L) \mid \Irr_{\ell'}(N)\cap \Irr(N\mid \Res^{\wt L}_L(\wt\la))\neq\emptyset\}.\]
That is, $\mathfrak{C}$ is the set of all  $\wt\la\in  \Irr(\wt L)$ satisfying that there is some $\chi\in\Irr_{\ell'}(N)$ lying above a constituent of the restriction $\Res^{\wt{L}}_L(\wt\la)$. Note that in this situation, $\chi$ lies over some $\la\in\Irr(L\mid \wt \la)$ such that $\la\in\Irr_{\ell'}(L)$ and $\ell\nmid [N:N_\la]$, by Clifford theory since $\la$ extends to $N_\la$ by \cite[Thm.~A]{spa09} and \cite[Thm.~1.1]{spath10}. We will also write \[\subL:=\{\la\in\Irr_{\ell'}(L)\mid \ell\nmid[N:N_\la]\}=\{ \lambda \in \Irr_{\ell'}(L) \mid \ell \nmid [W:W(\la)] \}.\]
where $W:=N/L$ and $W(\la):=N_\la/L$ for $\la\in\Irr(L)$.

It will be useful to note that $d$ is called a  \emph{regular number} for $(\bG, F)$ if $\bL$ is a torus. (See \cite[3.5.7]{GM20}.)  Moreover, we let $(\G^\ast,F)$ be a group in duality with $(\G,F)$ and denote $G^\ast:=(\G^\ast)^F$.

\subsection{A criterion for the inductive Isaacs--Navarro condition}

The following follows from Section \ref{sec:nonlie} and the results of \cite{RSST} and reduces us to determining appropriate extension maps and transversals.

\begin{theorem}\label{thm:RSST}
    Let $(\bG, F)$ be as above, keeping Notation \ref{not:lie type}, and such that $G=\bG^F$ is quasisimple. Assume further that all of the following hold:

   \begin{enumerate}
       \item\label{extmap} there exists an $(\Irr(\wt N/N)\rtimes  \widehat N\galh_0)$-equivariant extension map $\wt \Lambda$ for 
       $\mathfrak{C}$
       with respect to $\wt{L}\lhd \wt{N}$;
    \item\label{transversals} there exists an $(\galh_0\times \widehat{N})$-stable $\wt{G}$-transversal $\mathbb{T}$ of $\Irr_{\ell'}(G)$ and $(\galh_0\times \widehat{N})$-stable $\tilde{N}$-transversal $\mathbb{T}'$ of $\Irr_{\ell'}(N)$ and extension maps $\Phi_{glo}$ and $\Phi_{loc}$ for $\mathbb{T}$ and $\mathbb{T}'$ with respect to $G\lhd GE$ and $N\lhd \widehat{N}$;
    \item\label{extcond} $[\Phi_{glo}(\chi), \alpha]=[\Phi_{loc}(\psi), \alpha]$ for each $\alpha\in(\widehat N \galh_0)_\psi$ and each $\chi\in\mathbb{T}$ and $\psi\in\mathbb{T}'$ such that $\wt\Omega(\Irr(\wt{G})\mid \chi)=\Irr(\wt{N}\mid \psi)$, where $\wt\Omega$ is the map guaranteed by \cite[Thm.~B]{RSST}.
   \end{enumerate}
	Then the inductive Galois--McKay condition holds with respect to the subgroup $\galh_0$ 
    for the simple group $G/\Z(G)$. That is, the inductive Isaacs--Navarro condition holds for $G/\Z(G)$.
\end{theorem}

Note that by the results of Section \ref{sec:nonlie}, in the pursuit of proving the Isaacs--Navarro condition for all nonabelian simple groups, we may assume that 
we are in the situation of the hypotheses of Theorem \ref{thm:RSST}.

\begin{proof}[Proof of Theorem \ref{thm:RSST}]

As remarked above, note that we may assume that $G$ is the universal covering group of $G/\Z(G)$; that $G/\Z(G)$ is not isomorphic to an alternating or sporadic group; that $\ell$ is odd and $\bG$ is defined in characteristic $p\neq \ell$; and that there is a Sylow $\ell$-subgroup $D$ of $G$ such that $\N_G(D)\leq \N_G(\bS)$.

When combined with \cite[Thm.~B]{RSST}, condition (\ref{extmap}) yields an  $(\galh_0\ltimes (\Irr(\wt{G}/G)\rtimes (GE)_{\bS}))$-equivariant bijection
	$$\wt \Omega: \Irr(\wt{G}\mid \Irr_{\ell'}(G))\rightarrow\Irr(\wt N\mid \Irr_{\ell'}(N)),$$
	such that $\wt \Omega(\chi)$ and $\chi$ lie above the same character of $\Z(\wt{G})$ for each 
	 $\chi \in \Irr(\wt{G}\mid \Irr_{\ell'}(G))$.
Then using \cite[Cor.~3.5]{RSST}, Conditions (\ref{transversals}) and (\ref{extcond}) yield the inductive Galois--McKay condition with respect to $\galh_0$. 
\end{proof}

We can prove condition (\ref{extmap}) of Theorem \ref{thm:RSST} via providing an extension map for the characters of $L$:

\begin{lemma}\label{lem:lift ext map}
    Assume that there exists an $(\hat{N} \times \mathcal{H}_0)$-equivariant extension map $\Lambda$ for {an $(\hat N\times\galh_0)$-stable $\wt{L}$-transversal of} $\subL$
    with respect to $L\lhd N$. Then there exists an $(\Irr(\wt N/N)\rtimes  \widehat N\galh_0)$-equivariant extension map $\wt \Lambda$ for 
    $\mathfrak{C}$
    with respect to $\wt{L}\lhd \wt{N}$.
\end{lemma}

\begin{proof}
This follows from \cite[Prop.~2.3]{CS25}, taking $A:=\wt{N}$, $X:=\wt{L}$, $A_0:=N$, $X_0:=L$, $\hat{A}:=\wt{N}\hat{N}\galh_0$, and $\hat{A}_0:=\hat{N}\galh_0$. (Note that there the statement would yield an extension map for $\Irr(\wt L)$ given an extension map for a transversal of $\Irr(L)$, but the same proof applies when considering just our subsets $\mathfrak{C}$ and $\subL$.)
\end{proof}
   
\begin{remark}\label{rem:lift ext map}
Note that, in particular, the assumption of Lemma \ref{lem:lift ext map} is satisfied in the case that $\wt L$ is abelian (or when it is known that $\Res^{\wt L}_L(\lambda)$ is always irreducible) assuming just an $\hat{N} \times \mathcal{H}_0$-equivariant extension map for 
$\subL$ 
with respect to $L\lhd N$. 
    \end{remark}

\section{Towards the inductive condition for groups of Lie type}\label{sec:toward}

Throughout this section, we keep the situation  of Notation \ref{not:lie type}.

\subsection{The extension condition}
In this subsection, we consider part (\ref{extcond}) of Theorem \ref{thm:RSST}. Thanks to the results of \cite{RSST}, we will obtain the Isaacs--Navarro condition for type $\type{A}$ in Corollary \ref{cor:typeA} below.

In the case where $G$ is a quasi-simple group, any character $\chi \in \Irr_{\ell'}(G)$ necessarily has a trivial determinantal character, so that Lemma \ref{lem:extnsinitial_var} applies. On the other hand, in the local situation, the following lemma is helpful toward applying Lemma \ref{lem:extnsinitial_var}, though rather technical:

\begin{lemma}\label{lem:locextstructure}
Let $\hat{N}:=\N_{G E(\G^F)}(\Levi)$.
    There exists a subgroup $\hat{V} \leq \hat{N}$ such that $\hat{N}=L \hat{V}$ and $\hat{V} \cap L$ is a $2$-group with $\hat{V} \cap L \leq \Z(L)$. Moreover, there exists $\hat{E} \leq \hat{V}$ such that $\hat{N}=N \hat{E}$ and $\hat{E} \cap N$ is an $\ell'$-group. If $\bG$ is not of type $\tD$, then $\hat{E}$ centralizes $V:=\hat{V} \cap N$ while if $\bG$ is of type $\tD$ and $\Levi$ is not a torus, there exists a subgroup $V_{\tD}$ of $2$-power index in $V$ such that $\hat{E}$ centralizes $V_{\tD}$.
\end{lemma}

\begin{proof}

Let $\T_0$ be a maximally split torus of $\bG$ with Weyl group $W_0$ and let $V_0\leq \N_{\bG}(\T_0)$ be Tits' extended Weyl group \cite{tits}. (See also, e.g. \cite[Setting~2.1]{spa09} for more comments on this group). Recall that we have a surjective group homomorphism $\rho: V_0\to W_0$ with kernel $H_0 :=V_0\cap \T_0$ an elementary abelian $2$-group. Moreover, recall that $\rho$ has a canonical set-theoretic splitting $r:W_0 \to V_0$ and we denote by $\tilde{w}_0:=r(w_0)$ the image of the longest element $w_0 \in W_0$ under $r$. 

As introduced in \cite[Sec.~3]{spa09}, an element $v \in \N_G(\T_0)$ is called a Sylow $d$-twist for $(\G,F)$ if $\T_0^{vF}$ contains a Sylow $d$-torus of $\G^{vF}$. The twist $v$ is called good if $\rho(V_0^{vF})=\C_{W_0}(\rho(v) F)$. Given a suitable Sylow $d$-twist $v \in V_0$, we denote by $\hat{F}_p$ the image in $\G^{vF} \langle F_p \rangle / \langle vF \rangle$ of the automorphism acting as $F_p$ on $\G^{vF}$ such that $\G^{vF} \langle \hat{F}_p \rangle \cong \G^{F} \langle F_p \rangle$ given by conjugation with an element $g$ whose Lang image under $F$ is $v$, see \cite[Prop.~3.6]{S21D2}. If moreover, $\Gamma_0 \leq G \Gamma$ with $G \Gamma_0= G \Gamma$ and $[\Gamma_0,v]=1$ then by \cite[Prop.~3.6]{S21D2} we have $\G^{vF} \langle \hat{F}_p, \Gamma_0 \rangle \cong \G^{F} E$.

Suppose first that $\G$ is not of type $\tD$ and that $d$ is a regular number. (That is, $d$ is such that $\bL$ is a torus.) In \cite{CS19}, in the verification of the conditions of \cite[Thm.~4.3]{CS19}, it is shown that there exists some $v \in V_0$ such that $v$ is a good Sylow $d$-twist and $\N_{\bG^{vF}}(\T_0)=\T_0^{vF} V_0^{vF}$.

Since $\G$ is not of type $\tD$, it follows that any non-trivial graph automorphism $\gamma$ acts like the longest element $w_0$ on the Weyl group $W_0$. In particular, $V_0$ is centralized by $\gamma \tilde{w}_0$. Hence, we can define $\hat{V}_0:=V_0 \hat{E}_0$ with $\hat{E}_0=\langle \hat{F}_p, \gamma \tilde{w}_0 \rangle$ if $\bG$ is split and admits a non-trivial graph automorphism and $\hat{E}_0=\langle \hat{F}_p \rangle$ otherwise.
Then $\N_{\G^{vF} \hat{E}_0}(\T_0)=\T_0^{vF} \hat{V}_0^{vF}$.
In particular, the intersection $V_0^{vF} \cap \T_0^{vF}=H_0^{vF}$ is an elementary abelian $2$-group. 

For the last claim observe that by definition $\langle \hat{F}_p \rangle \cap \G^{vF}=\langle \hat{F} \rangle$ where $\hat{F}=v^{-1}$. 
As observed in \cite[Lem.~8.7]{RSST} the image of $v$ in $W$ has order dividing $\mathrm{lcm}(\delta,d)$ in $W_0$ (where $\delta \in \{1,2\}$ is the smallest positive integer such that $F^\delta$ acts trivially on $W_0$). As $V_0/H_0 \cong W_0$ and $H_0$ is a $2$-group, it follows that $v$ has order dividing $4d$. As $\ell \nmid 4d$ it thus follows that $\hat{E}_0 \cap \G^{vF}$ is an $\ell'$-group.

Assume now that $d$ is not regular. Let $\Phi$ be the root system of $\bG$ relative to the maximal torus $\T_0$. We let $\Levi_I$ be a standard Levi subgroup with root system $\Phi_I \subset \Phi$ such that a minimal $d$-split Levi subgroup of $(\G,F)$ is $\G$-conjugate to $\Levi_I$. Set $\Phi':=\Phi \cap \Phi_I^{\perp}$. Suppose first that $\G$ is of type $\mathrm{X}\in \{\tA,\tB,\tC\}$ and consider $\G_1:=\langle X_\al \mid \al \in \Phi' \rangle$ a simple simply connected group of type $\mathrm{X}_{n-r}$ such that $d$ is regular for $(\G_1,F)$. {(Here $X_\alpha$ denotes the root subgroup corresponding to $\alpha\in\Phi$).} Let $V_1$ be the extended Weyl group of $\bG_1$ relative to the maximal torus $\T_1:=\T_0 \cap \G_1$. We can take a good Sylow $d$-twist $v \in \N_{\bG}(\T_1)$ of $(\bG_1,F)$ such that $(\Levi_I,vF)$ is a minimal $d$-split Levi subgroup of $(\G,vF)$. It follows that $\N_{\G^{vF}}(\Levi_I)=\Levi_I^{vF} V_1^{vF}$ with $V_1 \cap \Levi_I \subset V_1 \cap H_0$ again an elementary $2$-group which is contained in $\Z(\Levi_I)$ by \cite[Bem.~2.1.7]{S07}. Moreover, by the results in the regular case, the element $v$ has order divisible by $4d$ which shows again that $\hat{E}_0 \cap \G^{vF}$ (with $\hat{E}_0$ defined as in the regular case) is an $\ell'$-group.

If the root system of $\bG$ is exceptional, the case where the Sylow $\ell$-subgroups of $\bG^F$ are non-cyclic follows from the discussion in \cite[Sec.~6.3]{CS19}.

So, assume that $\G$ is of exceptional type and the Sylow $\ell$-subgroups are cyclic. Note that when $G=\tE_7(q)$ and $d=4$ the Sylow $\ell$-subgroup of $G$ are non-cyclic. In the remaining cases, \cite[Table 3]{spa09} gives the precise structure of $N/L$. If $\bG$ is not of type $\tE_6$, then we see from this table that $N/L$ is cyclic of size $ \mathrm{lcm}(2,d)$. In particular, $\N_{\bG^{vF}}(\Levi)=\Levi^{vF} \langle \tilde{w}_0, v \rangle$, where $\tilde{w}_0 \in \Z(V_0)$ is the canonical preimage of the longest element of the Weyl group and $v \in V_0$ is a Sylow $d$-twist. Note that as argued in \cite[Lem.~8.6]{RSST}, we can always choose $v$ to satisfy $v^d \in H_0$. On the other hand, if $\Levi\neq \bG$ then $d\mid o(v H_0)$ which shows that $d=o(v H_0)$. In type $\tE_6(q)$ for $d=5$ (resp. $d=10$ in $\tE_6(-q)$) we have $\N_{\bG^{v F}}(\Levi_I)=\Levi_I^{v F} \langle v_0 \rangle$ for some element $v_0 \in V_0$ with $v_0^{5} \in H_0$ and $(vF)^{\delta}=v_0 F^{\delta}$ with $\delta \in \{1,2\}$. The claim follows from this in all cases.

Suppose now that $G$ is of type ${}^3 \tD_4(q)$. In this case $E=\langle F_p\rangle$ and so the claim follows directly from \cite[Satz 5.2.7]{S07}. If $G=\tD_4(q)$, then $d$ is only relevant when $d \in \{1,2,3,4,6 \}$. If $d=1,2$ then we can simply take $v\in \{1, \tilde{w}_0 \}$ and $\hat{E}_0=\langle \hat{F}_p, \Gamma \rangle$. Suppose that $d \in \{3,6\}$. In this case $V_0 \langle \hat{F}_p, \Gamma \rangle$ is an $\ell'$-group as $\ell \notin \{2, 3\}$. Note that a Sylow $3$-subgroup $\hat{V}_3$ of $V_0 \Gamma$ is isomorphic to a Sylow $3$-subgroup of $W_0 \Gamma$. A calculation in MAGMA shows that $\hat{V}_3 \cong C_3 \times C_3$. It follows that $1 \neq v \in \hat{V}_3 \cap V_0$ is a Sylow $3$-twist and $v \tilde{w}_0$ is a Sylow $6$-twist of $\G$. Moreover, we can assume that $v$ is $\gamma$-stable, where $\gamma$ is a graph automorphism of order $2$. We  therefore, take $\hat{V}_0:=\langle v, \tilde{w}_0 \rangle \hat{E}_0$ with $\hat{E}_0:=\langle v_3, \gamma, \hat{F}_p \rangle$ for some $v_3 \in \hat{V}_3 \setminus V_3$.

Suppose now that $G$ is of type ${}^\varepsilon \tD_n(q)$ and $d$ is doubly regular, i.e. $d \mid 2n$ such that $\frac{2n}{d}$ is even (resp. $\frac{2n}{d}$ is odd if $\varepsilon=2$). In this case the construction in \cite[Proposition 4.8]{CS25} yields the claim.

Assume now that $d$ is not doubly regular.
We embed $\G$ into a group $\overline{\G}$ of type $\tB_n$ of the same rank as in \cite[Not.~3.3]{S21D1}. We also make use of the Steinberg presentation of the group $\overline{\G}$ as discussed in \cite[Not.~3.3]{S21D1}. In particular, for $\overline{\al} \in \ov{\Phi}$ and $t \in \mathbb{F}^\ast$ we have the elements $h_\al(t) \in \T_0$ and $n_{\al}(t)\in \N_{\ov \bG}(\T_0)$ as defined in \cite[Not.~3.3]{S21D1}. 

Let $\overline{W}_0$ (resp. $\overline{V}_0$) be the Weyl group (resp. Tits' group) in $\overline{\G}$ relative to $\T_0$ containing $W_0$ (resp. $V_0$) as a subgroup of index $2$. Fix an element $\omega\in \mathbb{F}^\times$ of order $4_{p'}$. For $\overline{m} \in \overline{V}_0$ we define $m:=\overline{m}$ if $\overline{m} \in V_0$ and $m:=\overline{m} n_{e_1}(\omega)$ otherwise. In addition, we set $m':=\overline{m}$ if $\overline{m} \in V_0$ and $m'=\overline{m} n_{e_1}(1)$ otherwise.

By \cite[Lems.~10.2,~11.3]{S10a}, there exists $r > 0$ with $d \mid 2(n-r)$ and a good Sylow $d$-twist $\overline{v}$ of $\overline{\G}_1:=\langle X_{\overline{\al}} \mid \overline{\al} \in \overline{\Phi} \cap \langle e_{r+1}, \dots, e_{n} \rangle \}$, a group of type $\tB_{n-r}$, such that $v$ is a Sylow $d$-twist of $G$.

Let $\overline{V}_1$ be the extended Weyl group associated to the maximal torus $\T_0 \cap \overline{\G}_1$ of $\overline{\G}_1$. Note that in this case, the Chevalley relations (see \cite[Satz 2.1.6, Bem.~2.1.7]{S07}) show that $[\gamma,m]=1$ for all $\overline{m}\in \overline{V}_1$. For this recall from \cite[Def.~3.4]{S21D1} that $\gamma$ is given by conjugation action with $n_{e_1}(\omega)$.
We have $[n_{e_1}(1),\overline{m}] \in \langle h_0 \rangle$ with $[n_{e_1}(1),\overline{m}]=1$ if and only if $\rho(\overline{m}) \in W_0$. Moreover, $n_{e_1}(\omega)=h_{e_1}(\omega) n_{e_1}(1)$ and $[h_{e_1}(\omega),\overline{m}]=1$ while $[n_{e_1}(1),h_{e_1}(\omega)]=h_{e_1}(-1)$. Hence, $[\gamma,m]=1$ as required.

Similarly, the Chevalley relations together with the fact that $\overline{v}$ is a good Sylow $d$-twist for $\overline{\bG}$ then show that $\N_{\bG^{vF}}(\Levi_I)=\Levi_I^{vF} V_1$ where $V_1:=\langle m\mid \overline{m} \in\overline{V}_1 \rangle \cap \bG^{vF}$ (resp $V_1:=\langle m' \mid \overline{m} \in\overline{V}_1 \rangle \cap \bG^{vF}$ if $F(h_{e_1}(\omega))=h_{e_1}(\omega^{-1})$). This is because if $m\in \overline{V}_1$ with $[\overline{m},\overline{v}]=1$ then $[m,v]=1$ by the computations above. Note that $V_1 \cap \Levi \leq \langle h_{\alpha}(-1) \mid \alpha \in \overline{\Phi}\cap \{e_{r+1},\dots, e_n\} \rangle \langle h_{e_1}(\omega) \rangle$. If $F(h_{e_1}(\omega))=h_{e_1}(\omega^{-1})$, then $V_1 \cap \Levi \leq \langle h_{\alpha}(-1) \mid \alpha \in \overline{\Phi}\cap \{e_{r+1},\dots, e_n\} \rangle$ while otherwise $V_1 \leq V_0^{\gamma}$ so again $V_1 \cap \Levi_I \leq \langle h_{\alpha}(-1) \mid \alpha \in \overline{\Phi}\cap \{e_{r+1},\dots, e_n\} \rangle$. From this we deduce that $V_1 \cap \Levi_I  \leq \Z(\Levi_I)$.

Moreover in this case we can set $\hat{E}_0=\langle \hat{F}_p, \gamma \rangle$ if $F$ is split and $\hat{E}_0=\langle \hat{F}_p \rangle$ if $F$ is twisted. Finally observe that $V_{1,\tD}:=\langle \overline{m} \mid \overline{m} \in \overline{V}_1 \cap V_0 \rangle \cap V_1$ is centralized by $\hat{E}_0$ and has $2$-power index in $V_1$.
\end{proof}

\begin{corollary}\label{cor:structure sylow}
    Keep the notation of Lemma \ref{lem:locextstructure}. There exists a Sylow $\ell$-subgroup $D \in \Syl_\ell(N)$ such that $D=\Z(L)_\ell \rtimes V_0$ where $V_0 \cap L=1$ and $V_0 \in \Syl_\ell(V)$. Moreover, if $\bG$ is not of type $\tD$ or $d$ is not regular, then $\N_{\hat{G}}(D)=\N_{G}(D) \hat{E}$.
\end{corollary}

\begin{proof}
We first claim that the Sylow $\ell$-subgroup of $L$ is central in $L$. By the proof of \cite[Thm.~7.1]{CS13} this holds whenever $\ell \geq 5$ (resp. $\ell\geq 7$ if $G=\tE_8(q)$). If $\ell=3$, then $d\in\{1,2\}$ is regular, so $\bL$ is a torus, hence abelian (see \cite[Ex.~3.5.7]{GM20}).  Similarly, if $\ell=5$ and $G=\tE_8(q)$ then $d \in \{1,2,4 \}$ is again regular, so $\bL$ is torus. So, we see the claim in all these cases.

From this it follows that $D \cap L=\Z(L)_\ell$ and the first claim follows from Remark \ref{rem:split sylow}. For the second claim let $V_0 \in \Syl_\ell(V)$ (resp. $V_0 \in \Syl_\ell(V_{\tD})$ if $\bG$ is of type $\tD$ and $d$ is not regular) such that by Remark \ref{rem:split sylow} $D:=\Z(L)_\ell \rtimes V_0$ is a Sylow $\ell$-subgroup of $N$ with $V_0 \cap L=1$. As $\hat{E}$ centralizes $V_0$ it follows that $\N_{\hat{G}}(D)=\N_{G}(D) \hat{E}$.
\end{proof}

\begin{remark}\label{rem:cabanes}
In the case where $\ell \geq 5$ (resp. $\ell \geq 7$ if $G=\tE_8(q)$), the first part of the previous corollary could also be obtained from the more precise description of Sylow $\ell$-subgroups in \cite[Lem.~4.16]{CE99}.
Indeed, in this situation the group $D$ is Cabanes, i.e. has a unique maximal normal abelian subgroup. Note that $L=\C_G(\Z(L)_\ell)$ by \cite[Props.~13.16,~13.19,~22.6]{CE04}. From \cite[Prop.~22.13]{CE04} we then get that $\Z(L)_\ell$ is necessarily the maximal normal abelian subgroup of $D$. Moreover, by \cite[Lem.~4.16]{CE99} we have $D=\Z(L)_\ell \rtimes S$ for some $S \leq D$ and $S \cap \Z(L)_\ell=1$. 
\end{remark}

From Lemmas \ref{lem:extnsinitial_var} and \ref{lem:locextstructure}, we obtain the ``extension part" of the inductive condition with respect to $\galh_0$.

\begin{corollary}\label{cor:extensionpart}
Let $(\bG, F)$ be as in Theorem \ref{thm:RSST} and continue to keep the notation before. Let $\Irr_{\ell', ext}(G)$ and $\Irr_{\ell', ext}(N)$ denote the subset of $\Irr_{\ell'}(G)$, resp. $\Irr_{\ell'}(N)$, of characters that extend to their inertia group in $GE$, resp. $\hat N$. 
Then there are extension maps $\Phi_{glo}$ and $\Phi_{loc}$ for $\Irr_{\ell', ext}(G)$ with respect to $G\lhd GE$, resp. $\Irr_{\ell', ext}(N)$ with respect to $N\lhd \hat N$ such that $[\Phi_{glo}(\chi), \alpha]=1$ and $[\Phi_{loc}(\chi), \beta]=1$ for each $\chi\in\Irr_{\ell', ext}(G)$, $\psi\in\Irr_{\ell', ext}(N)$, $\alpha\in (\hat N \galh_0)_\chi$, and $\beta\in(\hat N\galh_0)_\psi$. In particular, {if (\ref{extmap}) and (\ref{transversals}) of Theorem \ref{thm:RSST} hold, then} (\ref{extcond}) of Theorem \ref{thm:RSST} holds.
\end{corollary}
\begin{proof}
Note that if $G\neq \type{D}_4(q)$, then the group $E$, and hence the group $\hat E$ from Lemma \ref{lem:locextstructure}, is abelian. If instead $G=\type{D}_4(q)$, then $E\cong S_3\times \langle F_p\rangle$ has a normal Sylow $\ell$-subgroup for $\ell\geq 3$. 

Let $\chi\in\Irr_{\ell'}(G)$ extend to $(GE)_\chi$ and let $\psi\in\Irr_{\ell'}(N)$ extend to $\hat N_\psi$. Since $G$ is perfect by assumption, the determinantal order of $\chi$ is 1.
On the other hand, Lemma \ref{lem:locextstructure} guarantees  the existence of $\hat K\lhd \hat N$ as in assumption (ii) in Lemma \ref{lem:extnsinitial_var} applied to $X:=N$ and $Y=\hat Y:=\hat N$, and we obtain the desired extensions from Lemma \ref{lem:extnsinitial_var}. 
\end{proof}

As pointed out in \cite{RSST}, the ``extension part" is what remains to be seen for type $\tA$ groups for the inductive Galois--McKay condition, so we now obtain the inductive Isaacs--Navarro condition in this case.

\begin{corollary}\label{cor:typeA}
  The inductive Isaacs--Navarro condition holds for  groups of Lie type $\tA$.  
\end{corollary}

\begin{proof}
In \cite[Secs.~5,~6]{RSST} it is shown that conditions (\ref{extmap}) and (\ref{transversals}) of Theorem \ref{thm:RSST} hold for groups of type $\tA$.  
By Corollary \ref{cor:extensionpart},   condition (\ref{extcond}) of Theorem \ref{thm:RSST} also holds. In particular, the assumptions of Theorem \ref{thm:RSST} are all satisfied and thus the inductive Isaacs--Navarro condition holds in this case.
\end{proof}

From Corollaries \ref{cor:extensionpart} and \ref{cor:typeA}, together with Theorem \ref{thm:RSST}, it follows that to prove Theorem \ref{thm:main}, it now suffices to find the equivariant extension maps with respect to $\wt{L}\lhd\wt{N}$ and the $\hat N\galh_0$-stable transversals described in (\ref{extmap}) and (\ref{transversals}) from Theorem \ref{thm:RSST}, for groups not of type $\tA$.

\subsection{The Transversals}

We next aim to obtain the transversals needed for part (\ref{transversals}) of Theorem \ref{thm:RSST}.

First, we recall here the notation $E:=E(\bG^F)$. Further, we recall that the irreducible characters of $G=\bG^F$ are partitioned into sets $\mathcal{E}(G, s)$, called rational Lusztig series, labeled by semisimple elements $s\in G^\ast$, up to $G^\ast$-conjugacy. In what follows we will often use that $\wt{G}E\times \galh_0$ permutes these series in a natural way (see, e.g. \cite[Prop.~7.2]{Tay} and \cite[Lem.~3.4]{SFT18}). In particular, if $\sigma\in \galh_0$, we have $\mathcal{E}(G, s)^\sigma=\mathcal{E}(G, s^k)$, where $\sigma$ acts as the $k$-power map on $|s|$-th roots of unity. 

We first obtain the required global transversal.

\begin{lemma}\label{lem:glob transversal}
    There exists an $(\mathcal{H}_0 \times E)$-stable $\tilde{G}$-transversal $\mathbb{T}$ of $\Irr_{\ell'}(G)$ such that every $\chi \in \mathbb{T}$ has an extension $\hat{\chi} \in \Irr(G E_\chi)$ with $(E \times \mathcal{H}_{0})_{\hat{\chi}}=(E \times \mathcal{H}_{0})_{\chi}$.
\end{lemma}

\begin{proof}
By \cite[Thm.~2.18]{CS25} there exists an $E$-stable $\tilde{G}$-transversal $\mathbb{T}_0$ of $\Irr(G)$ such that each $\chi\in\mathbb{T}_0$ extends to $GE_\chi$. To obtain our desired transversal, it suffices to show that every character $\chi \in \mathbb{T}_0\cap \Irr_{\ell'}(G)$ satisfies $(\tilde{G} \mathcal{H}_0 E)_\chi= \tilde{G}_\chi (\mathcal{H}_0 E)_\chi$. (Indeed, let $\mathcal{X}$ denote the set of characters $\chi$ such that $\chi$ extends to $GE_\chi$ and $(\tilde{G} \mathcal{H}_0 E)_\chi= \tilde{G}_\chi (\mathcal{H}_0 E)_\chi$. Then $\mathcal{X}$ is $\galh_0 E$-stable and for $\chi\in\mathcal{X}$, any two distinct characters in the $\galh_0E$-orbit of $\chi$ must lie in distinct $\wt{G}$-orbits. Further, if $\chi, \psi\in\mathcal{X}$ are in distinct $\wt{G}\galh_0E$-orbits, then their $\galh_0E$-orbits intersect distinct $\wt{G}$-orbits. With this, we are able to construct an $\galh_0E$-stable transversal such that each $\chi\in\mathbb{T}$ extends to $GE_\chi$, by taking $\mathbb{T}$ to be the union of $\galh_0E$-orbits of characters in $\mathbb{T}_0\cap \Irr_{\ell'}(G)$ from distinct $\wt{G}\galh_0E$-orbits, once we know $\mathbb{T}_0\cap \Irr_{\ell'}(G)\subseteq \mathcal{X}$.)

 Now, assume that $\sigma \in \mathcal{H}_{0}  E$ stabilizes the $\tilde{G}$-orbit of a character $\chi$ in $\mathbb{T}_0\cap\Irr_{\ell'}(G)$. We claim that $\chi$ is also $\sigma$-stable. If $\sigma \in E$ this follows from the fact that the transversal $\mathbb{T}_0$ is $E$-stable. We can therefore assume that $\sigma \notin E$. As $\mathcal{H}_0$ is an $\ell$-group, we can further assume that $\sigma$ has $\ell$-power order. (Write $\sigma=\sigma_\ell \sigma_{\ell'}$ where $\sigma_\ell$ is the $\ell$-part and $\sigma_{\ell'}$ the $\ell'$-part of $\sigma$. Then there is some $e \geq 0$ such that $\sigma^{\ell^e}=\sigma_{\ell'}$. Hence, $\sigma_{\ell'} \in E$ and still stabilizes the $\tilde{G}$-orbit of $\chi$. Thus, $\sigma_{\ell'}$ stabilizes $\chi$.) Let $E_\ell\in\Syl_\ell(E)$.

    Assume first that $\G$ is of type $\tA$. Then observe that any element of $\mathcal{H}_0$ acts trivially on $\ell'$-roots of unity; so in particular on $p$th roots of unity. Then the statement follows from Lemma \ref{lem:extnsinitial_var} and \cite[Thm.~5.8]{RSST}.

   Then we now assume that $\bG$ is not of type $\tA$. Assume first that $\ell \neq 3$ if $G$ is of type $\tE_6(\pm q)$ or of type $\tD_4(q)$. Observe that the length of the $\tilde{G}$-orbit of $\chi$ divides $|\tilde{G}/G \Z(\tilde{G})|$. As $\ell$ is odd, the latter is coprime to $\ell$. By coprimality, it therefore follows that the $\tilde{G}$-orbit of $\chi$ has a $\sigma$-fixed point. Note that in this case, the actions of $\tilde{G}$ and $\mathcal{H}_0 \times E_\ell$ on $\Irr(G)$ commute. 
   Then
    it follows that $\sigma$ fixes $\chi$ as well. 
    
    Next suppose that $\ell=3$ and $G=\tD_4(q)$.
    By the degree properties of Jordan decomposition, it follows that $\C_{G^\ast}(s)$ must contain a Sylow $3$-subgroup of $G^\ast$. Moreover, we may assume that $\C_{\G^\ast}(s)$ is disconnected as otherwise every character in $\mathcal{E}(G,s)$ is $\tilde{G}$-stable. We consider the list in \cite{Luebeck} of possible centralizers. Note that $\Phi_3 \Phi_6$ divides the polynomial order of $\tD_4(q)$. Therefore, the polynomial order of the centralizer of any $3$-central element must contain $\Phi_3$ or $\Phi_6$  (if $q\equiv 1 \mod 3$, resp. $q\equiv 2 \mod 3$) since $|\C_{\bG^\ast}(s)/\C_{\bG^\ast}^\circ(s)|$ is prime to $3$.  On the other hand, $\Phi_1^4$ resp. $\Phi_2^4$ divides the polynomial order of $\tD_4(q)$. These considerations show that only centralizers of type $\tA_3(q) \Phi_1.2$ resp. $\tw{2}\tA_3(q) \Phi_2 .2$ are relevant. These centralizers come from involutions $s \in G^\ast$. In particular, the corresponding Lusztig series $\mathcal{E}(G,s)$ is $\mathcal{H}_0$-stable. The characters in this Lusztig series are, up to diagonal automorphisms, determined by their degree. As $\mathcal{H}_0$ is a $3$-group and $A(s):=\C_{\bG^\ast}(s)/\C_{\bG^\ast}^\circ(s)$ has size $2$ (so that the diagonal automorphisms act with order two on these sets) it follows that every character in $\mathcal{E}(G,s)$ is necessarily $\mathcal{H}_0$-stable. This gives the claim in this case.

    It thus remains to consider the case where $\ell=3$ and $G=\tE_6(\pm q)$. Again 
    $\C_{G^\ast}(s)$ must contain a Sylow $3$-subgroup of $G^\ast$ and $\C_{\G^\ast}(s)$ can be assumed to be disconnected as otherwise every character in $\mathcal{E}(G,s)$ is $\tilde{G}$-stable. By considering the list in \cite{Luebeck} this shows that $s$ is quasi-isolated of order $3$ with centralizer of rational type $\tA_2(\pm q)^3.3$. Next we observe that for such an $s$, $s^k$ is $G^\ast$-conjugate to $s$ for each $k$ coprime to $3$. In particular, $\mathcal{E}(G, s)$ is $\galh_0$-stable and $\gamma$-stable, where $\gamma\in E$ is an order-two graph automorphism. By the degree properties of Jordan decomposition, the Jordan correspondent of $\chi$ in $\mathcal{E}(\C_{G^\ast}(s),1)$ must also have degree prime to $3$, so must lie above a character of $\C^\circ_{\G^\ast}(s)^F$ stable under the action of $\C_{\G^\ast}(s)^F/\C^\circ_{\G^\ast}(s)^F$, since the latter has size $3$. That is, the restriction to each copy of $\tA_2(\pm q)$ is the same. 
    In particular, since the unipotent characters of a group of type $\tA_2(\pm q)$ have distinct degrees, we see the character $\chi$ is uniquely determined, up to $\tilde{G}$-conjugation, by its degree in $\mathcal{E}(G,s)$. As $\gamma$ acts faithfully on $\tilde{G}/G \Z(\tilde{G})$ we therefore deduce that there exists a unique $\gamma$-stable character $\chi_0$ in the $\tilde{G}$-orbit of $\chi.$ As the actions of 
    $\galh_0 $ and $\gamma$ commute, $\galh_0$ must send $\gamma$-stable characters to $\gamma$-stable characters. Hence, $\chi_0$ must be $\galh_0$-stable. As $\chi$ is $\tilde{G}$-conjugate to $\chi_0$, it follows that $\chi$ is $\galh_0$-stable as well. This completes the claim.

    The last statement then immediately follows from Corollary \ref{cor:extensionpart}.
\end{proof}

We next obtain our desired local transversal. 

\begin{lemma}\label{lem:loc transversal}
Assume that $\bG$ is not of type $\tA$. Let $D\in\Syl_\ell(N)$ set $M=N$ if $d$ is regular and $M=\N_G(D)L$ otherwise.
    There exists an $\N_{\mathcal{H}_0 \times G E}(M)$-stable $\N_{\tilde{G}}(M)$-transversal $\mathbb{T}'$ of $\Irr_{\ell'}(M)$ such that every $\chi \in \mathbb{T}'$ has an extension $\hat{\chi} \in \Irr(\N_{G E_\chi}(M))$. {Further,  
    this extension can be chosen such that} $\N_{GE \times \mathcal{H}_{0}}(M)_{\hat{\chi}}=\N_{GE \times \mathcal{H}_{0}}(M)_{\chi}$.
\end{lemma}

\begin{proof}
    The inductive McKay condition, see \cite[Thm.~B]{CS25}, together with the transversal $\mathbb{T}$ from Lemma \ref{lem:glob transversal}, automatically gives an $\N_{G E }(M)$-stable $\N_{\tilde{G}}(M)$-transversal $\mathbb{T}'$ of $\Irr_{\ell'}(M)$ such that every $\chi \in \mathbb{T}'$ has an extension $\hat\chi$ to $\N_{G E_\chi}(M)$. Since $\bG$ is not of type $\tA$, we may argue just as in the proof of Lemma \ref{lem:glob transversal} to obtain the  $\mathcal{H}_0\N_{GE}(M)$-stable transversal, again using coprimality arguments when $\ell\neq 3$ or $G$ is not $\tE_6(\pm q)$ nor $\tD_4(q)$. 
    In the case of $\tE_6(\pm q)$ when $\ell=3$, using the $\Aut(G)_D$-equivariant bijection between $\Irr_{3'}(G)$ and $\Irr_{3'}(M)$ guaranteed by \cite[Thm.~B]{CS25}, we again see that the $3'$-degree characters of $M$ that are not $\N_{\wt{G}}(M)$-stable have a unique $\gamma$-stable $\N_{\wt{G}}(M)$-conjugate, so the same argument as in the proof of Lemma \ref{lem:glob transversal} applies. In the case of $\tD_4(q)$ when $\ell=3$, we similarly see using the equivariance from \cite[Thm.~B]{CS25} and the discussion in Lemma \ref{lem:glob transversal} of the characters in $\Irr_{3'}(G)$ that the $3'$-degree characters of $M$ that are not $\N_{\wt G}(M)$-stable are in $\N_{\wt{G}}(M)$-orbits of size two, and these orbits must be $\N_{GE}(M)$-stable. Then as in the proof of Lemma \ref{lem:glob transversal}, the characters must be $\galh_0$-stable. The last statement follows from using Lemma \ref{lem:locextstructure}, resp. Corollary \ref{cor:structure sylow}, together with Lemma \ref{lem:extnsinitial_var},  noting that $\N_{GE}(M)=M \hat{E}$ and $\hat{E} \cap M \leq \hat{E} \cap N$ is an $\ell'$-subgroup.
\end{proof}

\subsection{Extension maps}
With the previous sections, we are finally left to consider part (\ref{extmap}) of Theorem \ref{thm:RSST}.
We begin by recalling the following statement.

\begin{lemma}\label{lem:ext map equiv}
    There exists an $\hat{N}$-equivariant extension map $L \lhd N$ for an $\tilde{L}$-transversal $\mathbb{T}_0$ of $\Irr(L)$.
\end{lemma}

\begin{proof}
 For exceptional groups, this is contained in \cite[Lems.~8.1, 8.2, Prop.~9.2]{spa09}. For $\bG$ of type $\tA$, this is \cite[Prop.~5.9]{CS17a}. For type $\tC$, it is within the proof of \cite[Thm.~6.1]{CS17b}, and for type $\tB$ it is in \cite[Prop.~5.19]{CS19}.
 Finally, for type $\tD$, the statement follows from \cite[Prop.~4.8, Thm.~6.9]{CS25}.
\end{proof}

\begin{lemma}\label{lem:tildeL/LZ}
Assume that $\Levi\neq \bG$. Then $\wt L/L\Z(\wt L)$ has order dividing $2$.
\end{lemma}
\begin{proof}
Note that  $\tilde{L}/L \Z(\tilde{L})$ is a quotient of $\tilde{G}/G \Z(\tilde{G})$, so it remains to consider the case that $\bG$ is of type  $\tA$, $\tD$, or $\tE_6$. Applying \cite[Rem.~1.7.6]{GM20} to $\Levi$, we have $\wt L/L\Z(\wt L)\cong (\Z(\Levi)/\Z^\circ(\Levi))_F$, the group of $F$-coinvariants. We proceed by analyzing each case.

In type $\tA$, we have $\Z(\Levi)$ is connected, as noted in the proof of \cite[Prop.~5.6]{CS17a}. If $\bG=\tE_6$, then either $d$ is regular, hence $\Levi=\Z(\Levi)$ is a torus, or $d=5$ if $F$ is untwisted (resp. $d=10$ if $F$ is twisted) and $\Levi$ is of type $\tA_1$. In the latter case, as $\Z(\Levi)/\Z^\circ(\Levi)$ is both a quotient of $\Z([\Levi,\Levi]) \cong C_2$ and $\Z(\bG)/ \Z^\circ(\bG) \cong C_3$ it follows that $\Z(\Levi)$ is again connected.

Finally, if $\bG$ is of type $\tD$, we embed $\bG\leq \overline\bG$ into a group $\overline{\bG}$ of simply connected type $\tB$, as in \cite[Sec.~2.E]{CS25}, and we keep the notation there. Let $h_0=h_{e_1}(-1)=h_{e_n}(-1)$ be as in \cite[2.24]{CS25}, so that $\Z(\overline\bG)$ is generated by $h_0$. Here $\Levi$ is $\bG$-conjugate to a Levi subgroup $\Levi_I$ of type $\tD_m$ for some $m$, and  $\Z^\circ(\Levi_I)=\langle h_{e_i}(t) \mid i > m, t \in \bar{\mathbb{F}}_p^\times \rangle$. Hence, it follows that $h_0\in\Z^\circ(\Levi_I)$, forcing $(\Z(\Levi)/\Z^\circ(\Levi))_F$ to be of size at most $2$.
\end{proof}

 From Lemmas \ref{lem:ext map equiv} and \ref{lem:tildeL/LZ} we obtain the following regarding the set $\subL$ defined in Notation \ref{not:lie type}.

\begin{lemma}\label{lem:ext map prop}
   There exists an $(\mathcal{H}_0 \times \hat{N})$-stable $\tilde{L}$-transversal $\mathbb{T}_1$ of $\subL$
   such that every character $\la \in \mathbb{T}_1$ extends to its inertia group in $\hat{N}$. 
\end{lemma}

\begin{proof}
    By Lemma \ref{lem:ext map equiv},  there exists an $\hat{N}$-equivariant extension map $\Lambda$ with respect to $L \lhd N$ for an $\tilde{L}$-transversal $\mathbb{T}_0$ of $\Irr(L)$. Let $\la \in \mathbb{T}_0 \cap \Irr_{\ell'}(L)$ be such that $\ell \nmid |W:W(\la)|$. Set $\psi:=\Ind_{N_\la}^{N}(\Lambda(\la))$. Then $\hat{N}_\psi=N \hat{N}_\la$ and so $(\tilde{N}\hat{N})_\psi=\tilde{N}_\psi \hat{N}_\psi$. By Lemma \ref{lem:loc transversal} there exists an $\tilde{N}$-transversal of $\Irr_{\ell'}(N)$ such that every character in this transversal extends to its inertia group in $\hat{N}$.

    In particular, there exists an $\tilde{L}$-conjugate $\la':=\la^{\tilde{l}}$ of $\la$ such that the character $\psi':=\Ind_{N_\la}^{N}(\Lambda(\la))^{\tilde{l}}$ extends to a character $\hat{\psi}'\in \Irr(\hat{N}_{\psi'})$ and $(\tilde{N} \hat{N})_{\psi'}=\tilde{N}_{\psi'} \hat{N}_{\psi'}$. The claim of the lemma therefore follows when $\Z(\Levi)$ is connected as in this case we always have $\la'=\la$.

    Note that we may also assume $\Levi\neq\bG$, as otherwise the claim is part of Lemma \ref{lem:glob transversal}. We can therefore also assume that $G \not \cong \tD_4(q)$ as in this case either $\Levi$ is a torus or $\Levi=\bG$. Now note that $\hat{N}/N$ acts on $\tilde{L}/L \Z(\tilde{L})$ and by Lemma \ref{lem:tildeL/LZ}, the latter is either trivial or of order $2$.  This implies $[\hat{N},\tilde{L}] \subset L \Z(\tilde{L})$. By the properties of $\mathbb{T}_0$, this gives $(\tilde{L} \hat{N})_{\la'}=\tilde{L}_{\la'} \hat{N}_{\la'}$, since $\la'$ is $\wt L$-conjugate to $\la\in\mathbb{T}_0$. 

    Observe that for any $\hat{\lambda}' \in \Irr(\hat{N}_{\la'} \mid \Lambda(\la)^{\tilde{l}})$, there is some $e\geq 1$ such that $\Res^{\hat{N}_{\lambda'}}_{N_{\lambda'}}(\hat{\lambda}')=e \Lambda(\la)^{\tilde{l}}$. (Indeed, recall from above that $(\tilde{N}\hat{N})_\psi=\tilde{N}_\psi \hat{N}_\psi$ and $(\tilde{N}\hat{N})_{\psi'}=\tilde{N}_{\psi'} \hat{N}_{\psi'}$. Since $\psi$ and $\psi'$ are $\wt L$-conjugate, it follows that  $\hat{N}_\psi=\hat{N}_{\psi'}$. Now, by Clifford correspondence, $\Lambda(\la)^{\tilde{l}}$ is $\hat{N}_{\la'}$-stable if and only if $\psi'$ is $\hat{N}_{\la'}$-stable. However, $\hat{N}_{{\psi}'}=\hat{N}_\psi=N \hat{N}_{\lambda}=N \hat{N}_{\lambda'}$, so that indeed $\psi'$, hence $\Lambda(\la)^{\tilde{l}}$, is $\hat{N}_{\la'}$-stable.)

    Since $\hat{N}_{\psi'}=N \hat{N}_{\la'}$, we have $\hat{N}_{\la'}$ is the stabilizer in $\hat{N}_{\psi'}$ of $\hat\la'$, and it follows that $\hat{\psi}_0':=\Ind_{\hat{N}_{\la'}}^{\hat{N}_{\psi'}}(\hat{\lambda}')$ is an irreducible character of $\hat{N}_{\psi'}$ using Clifford correspondence. In particular, $\Res_{N}^{\hat{N}_{\psi'}}(\hat{\psi}'_0)=\Ind_{N_{\la'}}^{N}(e \Lambda(\la)^{\tilde{l}})=e \psi'$ by Mackey's formula. As $\hat{N}/N$ is abelian (recall that we can assume that $G \not \cong \tD_4(q)$) and $\psi'$ extends to $\hat{N}_{\psi'}$ it follows that $e=1$. In particular, $\hat{\lambda}'$ is an extension of $\Lambda(\la)^{\tilde{l}}$. Then we obtain an $\wt L$-transversal $\mathbb{T}_1$ of $\subL$ such that every character in $\mathbb{T}_1$ extends to its inertia group in $\hat{N}$. As $(\tilde{L} \hat{N})_{\la'}=\tilde{L}_{\la'} \hat{N}_{\la'}$ the so-obtained transversal is $\hat{N}$-stable.
 (Indeed, note that the set of characters satisfying the latter condition is $\hat N$-stable, and that no two distinct $\wt L$-conjugates with this property lie in the same $\hat N$-orbit.)

    Finally, note that the transversal $\mathbb{T}_1$ is automatically $\mathcal{H}_0$-stable, since $\tilde{L}/L \Z(\tilde{L})$ is a $2$-group by Lemma \ref{lem:tildeL/LZ}, and $\mathcal{H}_0$ is an $\ell$-group for an odd prime $\ell$ whose action commutes with that of $\wt{L}/L\Z(\wt{L})$ (see the proof of Lemma \ref{lem:glob transversal}).
\end{proof}

\subsection{The regular case}

Our next goal is to complete the proof of the inductive Isaacs--Navarro condition in the case that $d$ is a regular number for $\G$, in which case $L$ is abelian. We begin with the following, for which we work with $\Lin(L)$, without necessarily assuming that $L$ is abelian.

\begin{lemma}\label{lem:extlinear}
There exists an $(\hat{N} \times \mathcal{H}_0)$-equivariant extension map for the set 
$\Lin(L)\cap\subL$ with respect to $L\lhd \hat{N}$.
\end{lemma}

\begin{proof}
Recall from Lemma \ref{lem:locextstructure} that there exists a subgroup $\hat{V} \leq \hat{N}$ such that $\hat{N}=L \hat{V}$ and $H:=\hat{V} \cap L$ is an abelian $2$-group.

Let $\la \in \Irr(L)$ be a linear character satisfying $\ell\nmid[W:W(\la)]$ and let $\theta$ be its (irreducible) restriction to the $2$-group $H$. It suffices to find an extension $\Lambda(\la)$ of $\la$ to $\hat{N}_\la$ such that $(\hat{N} \times \mathcal{H}_{0})_{\Lambda(\la)}=(\hat{N} \times\mathcal{H}_{0})_{\la}$. {(Indeed, then we obtain a well-defined $(\hat N\times \galh_0)$-equivariant extension map $\Lambda'$ by defining $\Lambda'(\la):=\Lambda(\la)$ for $\la$ in some $(\hat N\times\galh_0)$-transversal, and $\Lambda'(\la^\sigma):=\Lambda(\la)^\sigma$ for any $\sigma\in \hat N\times \galh_0$.)}

By Lemma \ref{lem:ext map prop}, the character $\la$ extends to a character $\hat\la$ of $\hat{N}_\la$. 
We denote by $\hat{\theta}$ the restriction of $\hat\la$ to $\hat{V}_\la$. In particular, we can consider $\hat{\theta}$ as a character of $\hat{V}_\la/[\hat{V}_\la,\hat{V}_\la]$, which extends the character $\theta$ of $H/([\hat{V}_\la,\hat{V}_\la] \cap H)$. As $H$ is a $2$-group, we can also take an extension  of $\theta$ to $\hat{V}_\la/[\hat{V}_\la,\hat{V}_\la]$ that is trivial on the Sylow $r$-subgroups of $\hat{V}_\la/[\hat{V}_\la,\hat{V}_\la]$ for primes $r\neq 2$. Note that such an extension is $\galh_0$-invariant, as $\galh_0$ acts trivially on $\ell'$-roots of unity, hence on $2$-power roots of unity since $\ell$ is odd. Let $\mathfrak{X}$ denote the set of all such extensions of $\theta$ (that is, $\mathfrak{X}$ is the set of all extensions of $\theta$ to $\hat{V}_\la/[\hat{V}_\la,\hat{V}_\la]$ that are trivial on the Sylow $r$-subgroups for $r\neq 2$). By Gallagher's theorem, the size of $\mathfrak{X}$ is a power of $2$. Note also that $(\hat{V} \times \mathcal{H}_0)_\la/\hat{V}_{\la}$ acts on $\mathfrak{X}$. Arguing as in the second paragraph of Lemma \ref{lem:glob transversal}, we see that $(\hat{V} \times \mathcal{H}_0)_\la/\hat{V}_{\la}$ is an $\ell$-group, and therefore there is some such extension $\hat{\theta}'\in \mathfrak{X}$ that is $(\hat{V} \times \mathcal{H}_0)_\la$-invariant.

Hence, the unique extension $\Lambda(\la) \in \Irr(\hat{N}_\la \mid \la)$ which restricts to $\hat{\theta}'$ on $\hat{V}_\la$ (see Lemma \ref{lem:gluing}) is again $(\hat{N} \times \mathcal{H}_0)_\la$-stable.
\end{proof}

We now obtain the inductive Isaacs--Navarro condition in the case that $d$ is regular.

\begin{proposition}\label{prop:regular}
    Assume that $d=d_\ell(q)$ is regular for the group $(\G,F)$ and that $\bG^F$ is quasisimple. Then the inductive Isaacs--Navarro condition holds for $\G^F/\Z(\bG^F)$ and the prime $\ell$.
\end{proposition}

\begin{proof}
We check that all conditions in Theorem \ref{thm:RSST} are satisfied. By Lemma \ref{lem:extlinear} and Lemma \ref{lem:lift ext map} (see also Remark \ref{rem:lift ext map}), condition (\ref{extmap}) holds.
    Then conditions (\ref{transversals}) and (\ref{extcond}) of Theorem \ref{thm:RSST} are also satisfied by Lemma \ref{lem:loc transversal}, Lemma \ref{lem:glob transversal}, and Corollary \ref{cor:extensionpart}.  Hence the statement follows by Theorem \ref{thm:RSST}.
\end{proof}

\section{The non-regular case}\label{sec:nonreg}

 In this section, we complete the proof of our main results. We continue to keep the situation of Notation \ref{not:lie type}.
 
\subsection{An extension map in the non-regular case}

We now move to the case that $d=d_\ell(q)$ is non-regular. That is, $L$ is not a torus. We make this assumption throughout this subsection. Note that with this assumption, we have $\ell>3$ and 
 $D\in\Syl_\ell(G)$ can be chosen such that $\N_G(D)\leq N$, thanks to \cite[Thm.~5.14]{Ma07}. 
 
 In this case, the group $L\N_G(D)$ will serve the role of our intermediate group $M$ as in \cite[Thm.~3.4, Cor.~3.5]{RSST}, in place of $N$. 
We will write $\Irr(L)^D$ for the set of $D$-stable characters in $\Irr(L)$, and $\Irr_{\ell'}(L)^D$ for the intersection $\Irr(L)^D\cap \Irr_{\ell'}(L)$.

\begin{remark}\label{rem:diagonal defect}
    Note that if $\la \in \Irr(L)$ is $D$-stable then the same is true for all of its $\tilde{L}$-conjugates, since $[\tilde{L},D] \subset L$.
\end{remark}

Recall from Lemma \ref{lem:ext map prop} that there exists an $\tilde{L}$-transversal $\mathbb{T}_1$ of
$\subL$
such that every $\la$ in $\mathbb{T}_1$ has an extension to $\hat{N}_\la$. Further, note that $\Irr_{\ell'}(L)^D\subseteq \subL$ since $D\leq N_\la$ for any $D$-invariant $\la\in\Irr_{\ell'}(L)$.
By Remark \ref{rem:diagonal defect}, the intersection 
\[\mathbb{T}_D:=\mathbb{T}_1 \cap \Irr_{\ell'}(L)^D\]
is therefore an $\tilde{L}$-transversal of $\Irr_{\ell'}(L)^D$. Since $\mathbb{T}_1$ is $(\mathcal{H}_0 \times \hat{N})$-stable, we see $\mathbb{T}_D$ is further $(\N_{\hat{G}}(D)\times\mathcal{H}_{0})$-stable, where we recall $\hat{G}:=GE$.

Note also that as $\ell> 3$ it follows that $E$ has a normal $\ell$-complement $E_{\ell'} \lhd E$, and we will consider the group $L\N_{GE_{\ell'}}(D)\lhd L\N_{\hat{G}}(D)$.

\begin{proposition}\label{prop:equiv ext map NEW}
    Keep the situation above. Then there exists an $(\N_{\hat{G}}(D)\times\galh_0)$-equivariant extension map $\hat\La$ with respect to $L \lhd L \N_{G E_{\ell'}}(D)$ for the $(\N_{\hat{G}}(D) \times \mathcal{H}_0)$-stable $\tilde{L}$-transversal $\mathbb{T}_D$ of $\Irr_{\ell'}(L)^D$. 
\end{proposition}

\begin{proof}
Note that it again suffices to show that each $\la\in\mathbb{T}_D$ has an extension to $L\N_{GE_{\ell'}}(D)$ with the same stabilizer in $\N_{\hat G}(D)\times \galh_0$.  
Recall that, writing $\hat{M}:=L\N_{\hat{G}}(D)$ and $M=L\N_G(D)$, we have $\hat{M}=M \hat{E}\leq \hat N$ by Corollary \ref{cor:structure sylow}. Then we have $\hat{M}=L(\hat{V}\cap \hat{M})$, where $\hat{V}$ is as in Lemma \ref{lem:locextstructure}, and it follows that the assumptions of Lemma \ref{lem:extnsinitial_var}(ii) are satisfied for $X:=L$, $Y:=L \N_{GE_{\ell'}}(D)$, and $\hat{Y}:=L\N_{\hat{G}}(D)$. Then we complete the claim by applying Lemma \ref{lem:extnsinitial_var} to the transversal $\mathbb{T}_D$.
\end{proof}

\begin{corollary}\label{cor:equiv ext map}
    There exists an $(\N_{\hat{G}}(D)\times\galh_0)$-equivariant extension map $\La$ with respect to $L \lhd L \N_{G}(D)$ for an $(\N_{\hat{G}}(D)\times\galh_0)$-stable $\tilde{L}$-transversal $\mathbb{T}_D$ of $\Irr_{\ell'}(L)^D$. 
\end{corollary}

\begin{proof}

Let $\mathbb{T}_D$ be the transversal from Proposition \ref{prop:equiv ext map NEW}, so that every character $\psi$ in $\mathbb{T}_D$ has an $(\N_{\hat{G}}(D)\times\galh_0)_\psi$-invariant extension $\hat{\Lambda}(\psi) \in \Irr(L \N_{G E_{\ell'}}(D)_\psi \mid \psi)$. Hence, $\Res_{L \N_G(D)_\psi}^{L \N_{G E_{\ell'}}(D)_\psi}(\hat{\Lambda}(\psi))$ is $(\N_{\hat{G}}(D)\times\galh_0)_\psi$-invariant.
\end{proof}

\subsection{The inductive conditions in the non-regular case}

The proof of Theorem \ref{thm:main} below will use induction to deal with the case that $d$ is not regular. The following will allow us to achieve the inductive step.

\begin{theorem}\label{thm:induc}
Keep the situation of Notation \ref{not:lie type} and assume that $d$ is not regular for the group $(\G,F)$ and that $\bG$ is not of type $\tA$.
Suppose that the inductive Isaacs--Navarro condition holds for the simple groups involved in all groups whose order is smaller than the order of $G/\Z(G)$.  
Let $D\in\Syl_\ell(G)$.
    Then there exists an $(\Irr(\wt{G}/G)\rtimes(\Aut(\tilde{G})_D \times \mathcal{H}_0))$-equivariant bijection $\tilde{\Omega}:\Irr_{\ell'}(\tilde{G}) \to \Irr_{\ell'}(\tilde{L}\N_{\tilde{G}}(D))$.
\end{theorem}

\begin{proof}
Recall that we have a regular embedding yielding $G\lhd \wt G$.  
Note that we may assume that $\ell\neq 3$, since otherwise $d\in\{1,2\}$ is regular. In particular, $\ell$ does not divide $|\tilde{G}:G \Z(\tilde{G})|$ so that $\tilde{D}:=D \Z(\tilde{G})_\ell$ is a Sylow $\ell$-subgroup of $\tilde{G}$. Let $\tilde{\la} \in \Irr_{\ell'}(\tilde{L})^D$.
As established in  
\cite[Sec.~4.B]{RSST}, the characters of the relative Weyl group $W_{\tilde{\lambda}}:=\N_{\tilde{G}}(\tilde{\Levi},\tilde{\lambda})/\tilde{L}$ are all $\mathcal{H}_\ell$-invariant. 
By our assumption and applying \cite[Thm.~A]{NSV20}, the Isaacs--Navarro Galois conjecture holds  
for the group $W_{\tilde{\la}}$. This in particular implies that we have an $\mathcal{H}_0$-equivariant bijection $\Irr_{\ell'}(W_{\tilde{\la}}) \to \Irr_{\ell'}(\N_{\tilde{G}}(\tilde{D})_{\tilde{\la}} \tilde{L}/\tilde{L})$. (Note that the image $\tilde{D}\tilde{L}/\tilde{L}$ is a Sylow $\ell$-subgroup of $\N_{\tilde{G}}(\tilde{\Levi})/\tilde{L}$ and that $\N_{\tilde{G}}(\tilde{D})\tilde{L}/\tilde{L}$ is its normalizer in $\N_{\tilde{G}}(\tilde{\Levi})/\tilde{L}$.) 
From this, we see that every character in $\Irr_{\ell'}(\N_G(\tilde{D})_{\tilde{\la}} \tilde{L}/\tilde{L})$ is $\mathcal{H}_0$-invariant. 

On the other hand, the inductive McKay condition (which holds for all finite groups thanks to \cite[Thm.~B]{CS25}) yields a collection $f_{\tilde{\la}}:\Irr_{\ell'}(\N_{\tilde{G}}(\tilde{\Levi},\tilde{\la})/\tilde{L}) \to \Irr_{\ell'}(\N_{\tilde{G}}(\tilde{D})_{\tilde{\la}} \tilde{L}/\tilde{L})$ of $(\Aut(\N_{\tilde{G}}(\Levi))_{\tilde{D}} \times \mathcal{H}_0)$-equivariant bijections.
(For this, first fix an $(\Aut(\N_{\tilde{G}}(\Levi))_D \times \mathcal{H}_0)$-transversal of $\Irr_{\ell'}(\tilde{L})^D$. Then for every character $\tilde{\lambda}$ in this transversal we get by the inductive McKay condition an $\Aut(\N_G(\Levi,\tilde{\la}))_D$-equivariant map. Then extend the definition of $f_{\tilde{\lambda}}$ in an $( \Aut(\N_{\tilde{G}}(\Levi))_D \times \mathcal{H}_0)$-equivariant way, noting that $\mathcal{H}_0$ acts trivially on both sides of $f_{\tilde{\lambda}}$).

Let $\Lambda$ be the $(\mathcal{H}_0 \times \N_{\hat{G}}(D))$-equivariant extension map with respect to $L \lhd L \N_{G}(D)$ for the $\wt{L}$-transversal $\mathbb{T}_D \subset \Irr_{\ell'}(L)^D$ from Corollary \ref{cor:equiv ext map}. By \cite[Prop.~2.3]{CS25}, the map $\Lambda$ yields an $(\Irr(\N_{\tilde{G}}(D)/\N_{G}(D)) \rtimes (\N_{\hat{G}}(D) \times \mathcal{H}_0))$-equivariant extension map $\tilde{\La}$ for $\tilde{L} \lhd \tilde{L} \N_{\tilde{G}}(D)$ for the set $\Irr_{\ell'}(\tilde{L})^D$. 

By Gallagher's theorem, for $\tilde{\lambda} \in \Irr_{\ell'}(\tilde{L})^D$, there is a bijection
$$
\Irr_{\ell'}(\N_{\tilde{G}}(D)_{\tilde{\lambda}} \tilde{L}/\tilde{L}) \to \Irr_{\ell'}(\N_{\tilde{G}}(D) \tilde{L} \mid {\tilde{\lambda}})$$ given by
$$\eta \mapsto \Ind_{\N_{\tilde{G}}(D)_{\tilde{\lambda}} \wt L}^{\N_{\tilde{G}}(D) \wt L}(\tilde{\Lambda}({\tilde{\lambda}}) \eta).
$$
Let $\tilde{t} \in \tilde{L}^\ast$ be such that $\tilde{\lambda} \in \mathcal{E}(\tilde{L},\tilde{t})$ and denote by $\wt N_0$ the stabilizer in $\wt N$ of $\mathcal{E}(\tilde{L},\tilde{t})$.
Then the characters in $\Irr_{\ell'}(\tilde{G}) \cap \mathcal{E}(\tilde{G},\tilde{t})$ are in bijection with pairs $(\psi, \wt\eta)$, where $\psi\in\Irr_{\ell'}(\wt L)\cap \mathcal{E}(\wt L, \wt t)$, up to $\wt N_0$-conjugation, and $\wt\eta\in \Irr_{\ell'}(\wt{N}_{\psi}/\wt L)$. (See \cite[Prop.~7.3]{Ma07} and its adaption to $\wt{G}$ in \cite[Sec.~4]{CS13}. Note that $\psi$ in turn corresponds to some unipotent character in $\Irr_{\ell'}(\C_{\wt L^\ast}(\wt{t}))$.) Let $\chi_{\tilde{t},\tilde{\eta}}^{\psi}$ denote the character in $\Irr_{\ell'}(\tilde{G}) \cap \mathcal{E}(\tilde{G},\tilde{t})$ corresponding to $(\psi, \wt\eta)$.

On the other hand, we note that any member of $\Irr_{\ell'}(\N_{\wt{G}}(D)\wt L)$ must lie above a character $\wt\la\in\Irr_{\ell'}(\wt{L})^D$, applying Clifford correspondence. This shows that we have a bijection 
$$\wt\Omega\colon\Irr_{\ell'}(\tilde{G}) \to \Irr_{\ell'}(\N_{\tilde{G}}(D) \tilde{L})$$ given by
$$\chi_{\tilde{t},\tilde{\eta}}^{\psi} \mapsto \Ind_{\N_{\tilde{G}}(D)_\psi \wt L}^{\N_{\tilde{G}}(D) \wt L}(\tilde{\Lambda}(\psi) f_{{\psi}}(\wt\eta)).$$
By the equivariance properties of Jordan decomposition in the connected center case (see \cite[Thm.~3.1]{CS13} for the case of group automorphisms and the main result of \cite{SV20} for $\gal$, hence $\galh_0$) and the equivariance properties of $\tilde{\Lambda}$ and $f_{\tilde{\lambda}}$, it follows that this bijection is $(\Irr(\tilde{G}/G) \rtimes (\Aut(\tilde{G})_{D} \times \mathcal{H}_0))$-equivariant. 
\end{proof}

We now complete the proof of our main results.

\begin{proof}[Proof of Theorem \ref{thm:main}]
By Theorem \ref{thm:reduction}, it suffices to show that the inductive Galois--McKay condition holds with respect to $\galh_0$ for each finite non-abelian simple group. Let $S$ be a nonabelian simple group with universal covering group $G$. By Lemmas \ref{lem:nonlie} and \ref{lem:altspor} and Corollary \ref{cor:typeA}, we may assume that $\ell$ is odd, $G=\bG^F$ with $(\bG, F)$ a finite reductive group of simply connected type defined over $\bar{\mathbb{F}}_p$ with $p\neq \ell$,  $\Z(G)$ is a nonexceptional Schur multiplier for $S$, $\bG$ is not of type $\tA$, and that $\N_G(D)\leq N$, where $N$ is as in Notation \ref{not:lie type} and $D\in\Syl_\ell(G)$. By Proposition \ref{prop:regular}, we may assume that $d:=d_\ell(q)$ is not regular for $(\bG, F)$. As before, let $G\lhd\wt{G}$, obtained by a regular embedding, and note  that $\ell\nmid[\wt{G}:G]$ by our assumption on $d$.

We proceed by induction. Namely, suppose that  the inductive Isaacs--Navarro condition holds for the simple groups involved in all groups whose order is smaller than the order of $G/\Z(G)$. 
Then by Theorem \ref{thm:induc}, we have an $(\Irr(\wt{G}/G)\rtimes(\Aut(\tilde{G})_D \times \mathcal{H}_0))$-equivariant bijection $\tilde{\Omega}:\Irr_{\ell'}(\tilde{G}) \to \Irr_{\ell'}(\tilde{L}\N_{\tilde{G}}(D))$. 
Let $\wt M:=\tilde{L}\N_{\tilde{G}}(D)$ and $M:=L\N_G(D)$. Note then that $\N_{\wt{G}}(M)=\wt{M}$, and we can apply \cite[Thm.~3.4, Cor.~3.5]{RSST} with this choice of $M$, noting that the required transversal conditions hold thanks to Lemmas \ref{lem:loc transversal} and \ref{lem:glob transversal}.
\end{proof}

\section{Corollaries \ref{cor:NTconj} and \ref{cor:Hungconj} and Further Remarks}\label{sec:corollaries}

As noted in the introduction, Corollaries \ref{cor:NTconj} and \ref{cor:Hungconj} follow from Theorem \ref{thm:main}. We discuss these further here.

Namely, the equivalence of the first two items in Corollary \ref{cor:NTconj} was shown to follow from the Isaacs--Navarro Galois conjecture in \cite[pp.~342]{IN} (see also \cite[Thm.~9.12]{Nav18}), which means that portion of Corollary \ref{cor:NTconj} now follows from Theorem \ref{thm:main}. 
As noted also in the introduction, that the third item implies the first was proven in \cite[Thm.~B]{NT19}.  Hence, combining \cite[Thm.~B]{NT19} and Theorem \ref{thm:main}, we obtain the full Corollary \ref{cor:NTconj}.

Similarly, \cite[Conj.~1.1 and Conj.~2.3]{hung} follow from the Isaccas--Navarro Galois conjecture, as can be seen from the proof of \cite[Thm.~2.5]{hung}. Hence Corollary \ref{cor:Hungconj} again follows from Theorem \ref{thm:main}.

We also remark that, thanks to \cite[Thm.~8.3]{hung} and our Corollary \ref{cor:Hungconj}, the final conjecture in \cite{hung}, namely \cite[Conj.~1.4]{hung} (which would give a lower bound on $|\Irr_{\ell'}(G)|$ in terms of $\mathrm{exp}(D/D')$ and $\ell$) would be a consequence of the open conjecture \cite[Conj.~B3]{NT21}. Further, as noted in \cite[pp.~12]{NT21}, the conclusion of \cite[Thm.~B1]{NT21} (which would be implied by \cite[Conj.~B3]{NT21})  now also holds in the special case that $D/D'$ is elementary abelian, thanks to Corollary \ref{cor:NTconj}.

We close by mentioning a conjecture of  Malle--Mart{\'i}nez--Vallejo, \cite[Conj.~B]{MMV25}, which gives an explanation of the so-called ``$\ell$-rationality gap"---the phenomenon that there exist finite groups $G$ such that every $\chi\in\Irr_{\ell'}(G)$ that is $\sigma_1$-stable is in fact $\ell$-rational.  Namely, \cite[Conj.~B]{MMV25} conjecturally classifies such groups.  While this conjecture was proved to follow from the McKay--Navarro conjecture in \cite[Cor.~5.6]{MMV25}, we remark that it does \emph{not} follow just from the $\galh_0$-version, i.e. it does not follow from our Theorem \ref{thm:main}. This is because the characters with $\ell$-rationality levels in $\{0, 1\}$ are not distinguished by $\galh_0$ alone.

\end{document}